\documentclass[12pt]{article}
\usepackage{amsmath,amsfonts,amsthm,amssymb,mathrsfs,mathtools}
\usepackage{graphicx}
\usepackage[usenames,dvipsnames]{color}
\usepackage{float}
\usepackage{graphicx}
\usepackage{subfigure}
\usepackage{caption}
\usepackage[symbol]{footmisc}
\usepackage{cite}
\usepackage{url}
\usepackage{caption}
\usepackage{epstopdf}
\usepackage{hyperref}
\usepackage{ulem}
\usepackage{color}

\newcommand{\dint}{\displaystyle\int}

\theoremstyle{plain}
\newtheorem{theorem}{Theorem}[section]
\newtheorem{hy}{Assumption}[section]
\newtheorem{corollary}[theorem]{Corollary}
\newtheorem{lemma}[theorem]{Lemma}
\newtheorem{proposition}[theorem]{Proposition}

\theoremstyle{definition}
\newtheorem{definition}[theorem]{Definition}

\theoremstyle{remark}
\newtheorem{remark}[theorem]{Remark}

\numberwithin{equation}{section}
\numberwithin{theorem}{section}

\usepackage{a4wide}
\usepackage{geometry}
\begin{document}
	\renewcommand{\thefootnote}{\fnsymbol{footnote}}
	\begin{center}
		{\Large \textbf{Brownian Bridge with Random Length and Pinning Point for Modelling of Financial Information}} \\[12pt] 
		\textbf{Mohammed Louriki} \footnote[1]{Mathematics Department, Faculty of Sciences Semalalia, Cadi Ayyad University, Boulevard Prince Moulay Abdellah,	P. O. Box 2390, Marrakesh 40000, Morocco. 
			E-mail: \texttt{louriki.1992@gmail.com}}\footnote[2]{ Mathematics Department, Linnaeus University, Vejdesplats 7, SE-351 95 V\"axj\"o, Sweden.}
		\\[0pt]
	\end{center}
	
	\begin{abstract}
		In this paper, we introduce an extension of a Brownian bridge with a random length by including uncertainty also in the pinning level of the bridge. The main result of this work is that unlike for deterministic pinning point, the bridge process fails to be Markovian if the pining point distribution is absolutely continuous with respect to the Lebesgue measure. Further results include the derivation of formulae to calculate the conditional expectation of various functions of the random pinning time, the random pinning location, and the future value of the Brownian bridge, given an observation of the underlying process. For the specific case that the pining point has a two-point distribution, we state further properties of the Brownian bridge, e.g., the right continuity of its natural filtration and its semi-martingale decomposition. The newly introduced process can be used to model the flow of information about the behaviour of a gas storage contract holder; concerning whether to inject or withdraw gas at some random future time.
		
	\end{abstract}
	\smallskip
	\noindent 
	\textbf{Keywords:} Bayes theorem, Brownian bridges, stopping times, Markov processes, semi-martingale decomposition.\\
	\\ 
	\\
	\textbf{MSC 2010:} 60G15, 60G40, 60J25, 60F99.
	\begin{center}
		\section{Introduction}
		\label{Setion_1}
	\end{center}

	A stochastic process, obtained by conditioning a known process to start from an initial point at time zero and to arrive at a fixed point $z$ in the state space at a deterministic future time $r>0$, is called a bridge with deterministic length $r$ and pinning point $z$ associated with the given process. Many interesting examples are known. We mention, the Brownian, Gamma, Gaussian, L\'evy, and Markov bridges, see \cite{EY}, \cite{FG}, \cite{FPY}, \cite{GSV} and \cite{HHM}. Particularly, the Brownian bridge with deterministic length and pinning point plays a key role in many areas of statistics and probability theory and has become a well-known powerful tool in a variety of applications, for example, it appears as the large population limit of the cumulative sum process, when sampling randomly without replacement from a finite population, see \cite{R}. Moreover, it comes out in the limit for the normalized difference between a given distribution and its empirical law. It also plays a crucial role in the Kolmogorov-Smirnov test. Furthermore, it has many applications in finance, see, e.g., \cite{B}, \cite{K}, \cite{BS} and \cite{EW}.
	
	Using pathwise representations of the bridges with deterministic length, Bedini et al. \cite{BBE} and Erraoui et al. see, \cite{EL} and \cite{EHL}, recently introduced bridges with random length by substituting the deterministic length $r$ in the explicit expression of the bridge with the values of a random time $\tau$. Bridges with random length associated with L\'evy processes have been studied by Erraoui et al. by reference to their finite-dimensional distributions, see \cite{EHL(Levy)}. Moreover, in \cite{BBE} the authors consider a new approach to credit risk, in which the information about the default time of a financial company, i.e. the time of bankruptcy, is modelled, using a Brownian bridge starting from zero and conditioned to vanish when the default occurs. The extension \cite{EL} of \cite{BBE} presents bridges of random length associated with Gaussian processes.
	For studies of the gamma bridge with random length, see \cite{EHL}.
	
	More recent works have introduced uncertainty in the pinning level of the bridge. For example, in the modelling framework for accumulation processes presented in \cite{BHM} the aggregate claims process takes the form of a gamma bridge with random pinning point. Moreover, in \cite{EV} the authors considered the problem of optimal stopping of a Brownian bridge with an unknown
	pinning point.
	
	The Markov property of the Brownian bridge with respect to its natural filtration was proven for random length and deterministic pinning point, as well as for deterministic length and random pinning point, in \cite{BBE} and \cite{HHM}, respectively.
	
	In the current article we allow for uncertainty in both pinning level and time level of the Brownian bridge, and we call this process a "Brownian bridge with random length and pinning point". For a strictly random time $\tau$ and a random variable $Z$, the Brownian bridge $\zeta=(\zeta_t, t\geq 0)$ with random length $\tau$ and pinning point $Z$ is defined by:
	\begin{equation}
	\zeta_t=W_{t\wedge \tau}-\dfrac{t\wedge \tau}{\tau}W_{\tau}+\dfrac{t\wedge \tau}{\tau}Z,~ t \geq 0,\label{eqdefofzetaintr}
	\end{equation}
	where $W=(W_t,t\geq 0)$ is a Brownian motion and $W$, $\tau$ and $Z$ are independent. The main result of this paper is that the Markov property of the Brownian bridge with random length and pinning point depends on the nature of its pinning point in the following sense: If the probability distribution of its pinning point is discrete, the Brownian bridge possesses the Markov property. Otherwise, if the law of its pinning point is absolutely continuous with respect to the Lebesgue measure, the Brownian bridge fails to be Markovian. Further results include the derivation of formulae to calculate the conditional expectation of various functions of the random pinning time, the random pinning location, and the future value of the Brownian bridge, given an observation of the underlying process.
	
	As an application, we suggest an information-based approach to gas storage valuation, where the flow of information that motivates the holder of a gas storage contract to act at time $\tau$ by injecting or withdrawing gas, is modelled explicitly through the natural completed filtration $\mathbb{F}^{\xi}$ generated by an underlying information process $\xi=(\xi_t,t\geq 0)$. In our model this process is defined to be the Brownian bridge with random length pinned at a two-point random variable, that is, $\xi$ takes the form:
	\begin{equation}
	\xi_t=W_{t\wedge \tau}-\dfrac{t\wedge \tau}{\tau}W_{\tau}+\dfrac{t\wedge \tau}{\tau}\bigg(z_1\mathbb{I}_{\{X=0\}}+z_2\mathbb{I}_{\{X=1\}}\bigg),~ t \geq 0,\label{eqdefofxiintr}
	\end{equation}
	where $X$ follows the Bernoulli distribution. The intuitive idea here is that the holder of the gas storage contract chooses to do nothing while the Brownian
	bridge information process is away from the boundaries $z_1$ and $z_2$, alternatively that the holder decides to act (inject or withdraw) when the Brownian bridge absorbs at $z_1$ or $z_2$. This raises the question whether actions (injecting or withdrawal) can be foreseen by observing the evolution of the Brownian bridge information process. Answering this question relies on the theoretical properties of this process, such as, the Markov property, the right continuity of its natural filtration and its semi-martingale decomposition. For a deeper discussion of gas storage valuation we refer the reader to \cite{BJ}, \cite{CF} and \cite{CL}.
	
	The remainder of this article is structured as follows. In section 2  we provide preliminary facts that are used throughout the paper. In section 3 we define the Brownian bridge $\zeta$ with both random length $\tau$ and pinning point $Z$. Then we analyse the conditions under which the Markov property of $\zeta$, with respect to its natural filtration $\mathbb{F}^{\zeta}$, holds. We prove that for a pinning point with discrete distribution, $\zeta$ is an $\mathbb{F}^{\zeta}$-Markov process, whereas for a pinning point, the law of which is absolutely continuous with respect to the Lebesgue measure, the process $\zeta$ cannot be an $\mathbb{F}^{\zeta}$-Markov process. Section 4 deals with the case, where the pinning point is two-point distributed. We show that the random length $\tau$ is an $\mathbb{F}^{\xi,c}$-stopping time, where $\mathbb{F}^{\xi,c}$ is the completed natural filtration of $\xi$. In addition to that, we prove that $\mathbb{F}^{\xi,c}$ satisfies the usual conditions of right-continuity and completeness. Finally, we derive the semi-martingale decomposition of $\xi$.
	
	The following notation will be used throughout the paper: 
	For a complete probability space $(\Omega,\mathcal{F},\mathbb{P})$, $\mathcal{N}_p$ denotes the
	collection of $\mathbb{P}$-null sets. If $\theta$ is a random variable, then $\mathbb{P}_{\theta}$ and $F_{\theta}$ are its law and its distribution function under $\mathbb{P}$, respectively. $C\left(\mathbb{R}_+,\mathbb{R}\right)$ denotes the canonical space, that is the space of continuous real-valued functions defined on $\mathbb{R}_+$, $\mathcal{C}$ the $\sigma$-algebra generated by the canonical process. If $E$ is a topological space, then the Borel $\sigma$-algebra over $E$ will be denoted by $\mathcal{B}(E)$. The characteristic function of a set $A$ is written $\mathbb{I}_{A}$. 
	The symmetric difference of two sets $A$ and $B$ is denoted by $A\Delta B$. $p(t, x, y)$, $x, y\in \mathbb{R}$, $t\in\mathbb{R}_+$, denotes the Gaussian density function with variance $t$ and mean $y$, if $y=0$, for simplicity of notation we write $p(t,x)$ rather than $p(t, x, 0)$. $\text{Cov} (Y_s,Y_t)$, $s,t \in \mathbb{R}_+$ is the covariance function associated with the process $Y$.
	Finally for any process $Y=(Y_t,\, t\geq 0)$ on $(\Omega,\mathcal{F},\mathbb{P})$, we define by:
	\begin{enumerate}
		\item[(i)] $\mathbb{F}^{Y}=\bigg(\mathcal{F}^{Y}_t:=\sigma(Y_s, s\leq t),~ t\geq 0\bigg)$ the natural filtration of the process $Y$.
		\item[(ii)] $\mathbb{F}^{Y,c}=\bigg(\mathcal{F}^{Y,c}_t:=\mathcal{F}^{Y}_t\vee \mathcal{N}_{P},\, t\geq 0\bigg)$ the completed natural filtration of the process $Y$.
		\item[(iii)] $\mathbb{F}^{Y,c}_{+}=\bigg(\mathcal{F}^{Y,c}_{t^{+}}:=\underset{{s>t}}\bigcap\mathcal{F}^{Y,c}_{s}=\mathcal{F}^{Y}_{t^{+}}\vee \mathcal{N}_{P},\, t\geq 0\bigg)$ the smallest filtration containing $\mathbb{F}^{Y}$ and satisfying the usual hypotheses of right-continuity and completeness.
	\end{enumerate}
	\begin{center}
		\section{Preliminaries}
	\end{center}
	We start by recalling some basic results on Brownian bridges and properties of conditional expectations that will be used in the sequel.\\
	\subsection{Brownian Bridge Processes}
	The Brownian bridge is a fundamental process in statistics and probability theory. This section summarizes a few well-known results about the extended Brownian bridge (a Brownian bridge defined for $t\in \mathbb{R}_{+}$).
	\begin{definition}
		Given a strictly positive real number $r$, a real number $z$ and a Brownian motion $W$, the process $ \zeta^{r,z}:\Omega \longmapsto C\left(\mathbb{R}_+,\mathbb{R}\right)$, defined by
		\begin{equation}
		\zeta_t^{r,z}(\omega):=W_{t\wedge r}(\omega)-\dfrac{t\wedge r}{r}W_{r}(\omega)+\dfrac{t\wedge r}{r}z,~ t \geq 0,~ \omega \in \Omega,\label{eqdefofzeta^r,z}
		\end{equation}
		is called Brownian bridge with deterministic length $r$ and pinning point $z$ associated to $W$.
	\end{definition}
	\begin{remark}\label{remarkmeasurable}
		The process $\zeta^{r,z}$ is the Brownian bridge, which is identically equal to $z$ on the time interval $[r,\infty[$. The process $\zeta^{r,z}$ is in fact a function of the variables $(r,t,z,\omega)$ and for technical reasons, it is convenient to have some joint measurability properties. Since the map $(r,t,z)\longmapsto \zeta_t^{r,z}(\omega)$ is continuous for all $\omega \in \Omega$, the map $(r,t,z,\omega)\longmapsto \zeta_t^{r,z}(\omega)$ of $\big((0,+\infty)\times \mathbb{R}_{+}\times \mathbb{R} \times \Omega ,\mathcal{B}\big((0,+\infty)\big)\otimes \mathcal{B}(\mathbb{R}_{+})\otimes \mathcal{B}(\mathbb{R})\otimes\mathcal{F}\big)$ into $(\mathbb{R}_{+},\mathcal{B}(\mathbb{R}_{+}))$ is measurable. In particular, the $t$-section of
		$(r,t,z,\omega)\longmapsto \zeta_t^{r,z}(\omega)$: $(r,z,\omega)\longmapsto \zeta_t^{r,z}(\omega)$ is measurable with respect to the $\sigma$-algebra
		$\mathcal{B}\big((0,+\infty)\big)\otimes \mathcal{B}(\mathbb{R})\otimes\mathcal{F}$, for all $t \geq 0$.
	\end{remark}
	The next proposition lists some useful properties of the process $\zeta^{r,z}$.
	\begin{proposition}
		The process $\zeta^{r,z}$ is a Gaussian process. Moreover, it satisfies the following properties:
		\begin{enumerate}
			\item[(i)] For all $0<t<r$, the random variable $\zeta^{r,z}_t$ is non-degenerate and its density function is given by: 
			\begin{equation}
			\varphi_{\zeta^{r,z}_t}(x)=p\bigg(\dfrac{t(r-t)}{r},x,\dfrac{t}{r}z\bigg).\label{zeta^{r,z}density}
			\end{equation}
			\item[(ii)] Let $n$ be an integer greater than $1$ and $0<t_1<t_2<...<t_n<r$, then the Gaussian vector $(\zeta_{t_1}^{r,z},\ldots,\zeta_{t_n}^{r,z})$ has an absolutely continuous density with respect to Lebesgue measure on $\mathbb{R}^n$. Moreover, its density function is given by 
			\begin{equation}
			\varphi_{\zeta_{t_1}^{r,z},\ldots,\zeta_{t_n}^{r,z}}(x_1,\ldots,x_n)=\dfrac{p(r-t_n,z-x_n)}{p(r,z)}\prod_{i=1}^{n}p(t_i-t_{i-1},x_i-x_{i-1}).\label{eqdensityofvectorzeta}
			\end{equation}
			for every $(x_1,x_2,...,x_n) \in \mathbb{R}^n$, with the understanding that $x_0 =t_0 =0$.
			\item[(iii)] The process $\zeta^{r,z}$ is a Markov process with respect to its completed natural filtration. Moreover, for all $0<t<u<r$, its transition density is given by: 
			\begin{align}
			\mathbb{P}(\zeta_u^{r,z}\in \mathrm{d}y\vert \zeta_t^{r,z}=x)
			&=\dfrac{p(r-u,z-y)p(u-t,y-x)}{p(r-t,z-x)}\mathrm{d}y.\label{eqtransitionlawofX^{r,z}lévy}
			\end{align}
			An equivalent formulation of \eqref{eqtransitionlawofX^{r,z}lévy} is 
			\begin{equation}
			\mathbb{P}(\zeta_u^{r,z}\in \mathrm{d}y\vert \zeta_t^{r,z}=x)=p\bigg(\dfrac{r-u}{r-t}(u-t),y,\dfrac{r-u}{r-t}x+\dfrac{u-t}{r-t}z\bigg)\mathrm{d}y.\label{eqtransitionlawofX^{r,z}}
			\end{equation}
			\item[(iv)] The process $\zeta^{r,z}$ satisfies the following equation 
			\begin{equation}
			\zeta_{t}^{r,z}=b_{t}^{r,z}+\int_{0}^{t}\dfrac{z-\zeta_{s}^{r,z}}{r-s}\mathbb{I}_{\{s<r\}}\mathrm{d}s,\label{eqsemimartingalezeta^r,zonR+}
			\end{equation}
			where $(b^{r,z}_t, t\geq 0)$ is an $\mathbb{F}^{\zeta^{r,z}}$-Brownian motion stopped at $r$.		
		\end{enumerate}
	\end{proposition}
	\begin{proof}
		\begin{enumerate}
			
			\item[(i)] The proof of the statement (i) is straightforward.
			\item[(ii)] For any bounded functional $F$ we have 
			\begin{equation}
			\mathbb{E}[F(W_t, t\leq r)|W_r=z]=\mathbb{E}[F(\zeta_t^{r,z}, t\leq r)].\label{eqbridgedefinitionfunctional}
			\end{equation}
			Using \eqref{eqbridgedefinitionfunctional} together with the fact that 
			\begin{equation}
			\mathbb{P}(W_{t_1}\in dx_1,\ldots,W_{t_n}\in dx_n)=\prod_{i=1}^{n}[p(t_i-t_{i-1},x_i-x_{i-1})\mathrm{d}x_i],
			\end{equation}
			for every $0<t_1<t_2<...<t_n$ and $(x_1,x_2,...,x_n) \in \mathbb{R}^n$, we obtain that
			\begin{align*}
			\mathbb{P}(\zeta^{r,z}_{t_1}\in \mathrm{d}x_1,\ldots,\zeta^{r,z}_{t_n}\in \mathrm{d}x_n )&=\mathbb{P}(W_{t_1}\in \mathrm{d}x_1,\ldots,W_{t_n}\in \mathrm{d}x_n|W_r=z )\nonumber\\
			&=\dfrac{p(r-t_n,z-x_n)}{p(r,z)}\prod_{i=1}^{n}[p(t_i-t_{i-1},x_i-x_{i-1})\mathrm{d}x_i].
			\end{align*}
			\item[(iii)] We only need to show that for every $0<t_1<t_2<...<t_n<u<r$ and $(x_1,x_2,...,x_n,y) \in \mathbb{R}^{n+1}$,
			\begin{equation}
			\mathbb{P}(\zeta_u^{r,z}\in \mathrm{d}y|\zeta^{r,z}_{t_1}=x_1,\ldots,\zeta^{r,z}_{t_n}=x_n )=\mathbb{P}(\zeta_u^{r,z}\in \mathrm{d}y|\zeta^{r,z}_{t_n}=x_n ).
			\end{equation}
			By using the statement (ii) we have,
			\begin{align}
			\mathbb{P}(\zeta_u^{r,z}\in \mathrm{d}y|\zeta^{r,z}_{t_1}=x_1,\ldots,\zeta^{r,z}_{t_n}=x_n )&=\dfrac{\varphi_{\zeta_{t_1}^{r,z},\ldots,\zeta_{t_n}^{r,z},\zeta_{u}^{r,z}}(x_1,\ldots,x_n,y)}{\varphi_{\zeta_{t_1}^{r,z},\ldots,\zeta_{t_n}^{r,z}}(x_1,\ldots,x_n)}\mathrm{d}y\nonumber\\
			\nonumber\\
			&=\dfrac{p(r-u,z-y)p(u-t_n,y-x_n)}{p(r-t_n,z-x_n)}\mathrm{d}y\nonumber\\
			\nonumber\\
			&=\mathbb{P}(\zeta_u^{r,z}\in \mathrm{d}y\vert \zeta_{t_n}^{r,z}=x_n).
			\end{align}
			Then $\zeta^{r,z}$ is a Markov process with transition law given by \eqref{eqtransitionlawofX^{r,z}lévy}. The proof is completed by showing the formula \eqref{eqtransitionlawofX^{r,z}}, which is an immediate consequence of the fact that
			\[
			(\zeta_u^{r,z}|\zeta_t^{r,z}=x)\stackrel{law}{=}(W_u|W_t=x,W_r=z)
			\]
			and the result of conditioning Brownian motion at time $u$ on the knowledge of its value at both an earlier and later time that is given by 
			\[
			(W_u|W_t=x,W_r=z)\stackrel{law}{=}\mathcal{N}\bigg(\dfrac{r-u}{r-t}x+\dfrac{u-t}{r-t}z,\dfrac{r-u}{r-t}(u-t)\bigg)
			\]
			where $t<u<r$ and $\mathcal{N}(\mu,\sigma^2)$ is the normal distribution with expectation $\mu$ and variance $\sigma^2$.
			\item[(iv)] From Corollary 4.1 in \cite{A}, the process $\zeta^{r,z}$ satisfies the following equation 
			\begin{equation}
			\zeta_{t}^{r,z}=B_{t}^{r,z}+\int_{0}^{t}\dfrac{z-\zeta_{s}^{r,z}}{r-s}\mathrm{d}s, \quad t\in[0, r],\label{eqsemimartingalezeta^r,z}
			\end{equation}
			where $(B^{r,z}_t, t\in[0, r])$ is an $\mathbb{F}^{\zeta^{r,z}}$-Brownian motion.
			Since $\zeta_{t\wedge r}^{r,z}=\zeta^{r,z}_t$, the formula \eqref{eqsemimartingalezeta^r,z} shows that   
			\begin{equation*}
			\zeta_{t}^{r,z}=b_{t}^{r,z}+\int_{0}^{t}\dfrac{z-\zeta_{s}^{r,z}}{r-s}\mathbb{I}_{\{s<r\}}\mathrm{d}s,
			\end{equation*}
			where $(b^{r,z}_t, t\geq 0)$ is an $\mathbb{F}^{\zeta^{r,z}}$--Brownian motion stopped at $r$.	
		\end{enumerate}	
	\end{proof}
	\subsection{Conditional Law}
	In this part, we recall some important results on conditional probabilities which will be used later.
	\begin{lemma}
		Let $U, V$ and $S$ be three random variables. The conditional law of $(U,V)$ given $S$ can be expressed as follows
		\begin{equation}
		\mathbb{P}_{(U,V)|S=s}(\mathrm{d}u,\mathrm{d}v)=\mathbb{P}_{U|V=v,S=s}(\mathrm{d}u)\mathbb{P}_{V \vert S=s}(\mathrm{d}v).\label{eqargumentconditioanal}
		\end{equation}
	\end{lemma}
	\begin{proof}
		Our proof starts with the observation that the law of triple of random variables $(U,V,S)$ can be found by two methods
		\begin{align}
		\mathbb{P}_{(U,V,S)}(\mathrm{d}u,\mathrm{d}v,\mathrm{d}s)&=\mathbb{P}_{(U,V)|S=s}(\mathrm{d}u,\mathrm{d}v)\mathbb{P}_{S}(\mathrm{d}s).\label{U,V,W,1}\\
		\nonumber\\
		\mathbb{P}_{(U,V,S)}(\mathrm{d}u,\mathrm{d}v,\mathrm{d}s)&=\mathbb{P}_{U|V=v,S=s}(\mathrm{d}u)\mathbb{P}_{(V,S)}(\mathrm{d}v,\mathrm{d}s)\nonumber\\
		&=\mathbb{P}_{U|V=v,S=s}(\mathrm{d}u)\mathbb{P}_{V \vert S=s}(\mathrm{d}v)\mathbb{P}_{S}(\mathrm{d}s).\label{U,V,W,2}
		\end{align}
		The proof is an immediate consequence of the two formulas \eqref{U,V,W,1} and \eqref{U,V,W,2}.
	\end{proof}
	\begin{lemma}[Bayes formula]
		Let $\theta$ be a random variable on a probability space $(\Omega,\mathcal{F},\mathbb{P})$ and $\mathcal{G}\subset \mathcal{F}$ be a $\sigma$-algebra. Suppose that for any $B\in \mathcal{G}$, the conditional probability $\mathbb{P}(B\vert \theta=a)$ is regular and admits the representation
		$$\mathbb{P}(B\vert \theta=a)=\dint_{B}\rho(\omega,a)\mu(\mathrm{d}\omega)$$
		where $\rho$ is non-negative and measurable in the two variables jointly, and $\mu$ is a $\sigma$-finite measure on $(\Omega,\mathcal{G})$. Then for every bounded measurable function $g$, we have 
		\begin{equation}
		\mathbb{E}[g(\theta)\vert \mathcal{G}](\omega)=\dfrac{\dint_{\mathbb{R}}g(a)\rho(\omega,a)\mathbb{P}_{\theta}(\mathrm{d}a)}{\dint_{\mathbb{R}}\rho(\omega,a)\mathbb{P}_{\theta}(\mathrm{d}a)}.
		\end{equation}
	\end{lemma}
	\begin{proof}
		See, \cite[pp. 272-274]{S}
	\end{proof}
	\begin{remark}\label{remarkBayes}
		Let $n$ be an integer greater than or equal to $1$ and $X$ be a random variable defined on a probability space $(\Omega,\mathcal{F},\mathbb{P})$ with values in $(\mathbb{R}^n, \mathcal{B}({\mathbb{R}^n}))$. Suppose that 
		$$\mathbb{P}(X\in A\vert \theta=a)=\dint_{A}q(x,a)\mu(\mathrm{d}x),\,\, A\in \mathcal{B}({\mathbb{R}^n}),$$
		where $q$ is a non-negative function, measurable with respect to both variables jointly, and $\mu$ is a $\sigma$-finite measure on $(\mathbb{R}^n, \mathcal{B}({\mathbb{R}^n}))$. Then we obtain
		\begin{equation}
		\mathbb{E}[g(\theta)\vert X=x]=\dfrac{\dint_{\mathbb{R}}g(a)q(x,a)\mathbb{P}_{\theta}(\mathrm{d}a)}{\dint_{\mathbb{R}}q(x,a)\mathbb{P}_{\theta}(\mathrm{d}a)}.
		\end{equation}
	\end{remark}
	\begin{center}
		\section{Brownian Bridges with Random Length and Pinning Point}\label{Brownian bridges with random length and an unknown pinning point}
	\end{center}
	In this section, we define a new process $(\zeta_t, t\geq 0)$, which generalizes the Brownian bridge in the sense that the time $r$, at which the bridge is pinned, and the pinned value $z$ of the bridge are substituted by a random time $\tau$ and a random variable $Z$, respectively. The aim of this section is to analyse the Markov property of $(\zeta_t, t\geq 0)$ with respect to its natural filtration.\\
	\begin{definition}
		Let $\tau: (\Omega,\mathcal{F},\mathbb{P}) \longmapsto (0,+\infty)$ be a strictly positive random time, $Z: (\Omega,\mathcal{F},\mathbb{P}) \longmapsto \mathbb{R}$ be a random variable and $W$ be a Brownian motion. The Brownian bridge with length $\tau$ and pinning point $Z$ associated to $W$, is defined by
		$$\zeta_{t}(\omega):=\zeta_{t}^{r,z}(\omega)\vert_{r=\tau(\omega)}^{z=Z(\omega)}~~, (t,\omega) \in \mathbb{R}_{+} \times \Omega.$$
		Combining the previous equality with \eqref{eqdefofzeta^r,z} reveals that $\zeta$ takes the form
		\begin{equation}
		\zeta_t(\omega)=W_{t\wedge \tau(\omega)}(\omega)-\dfrac{t\wedge \tau(\omega)}{\tau(\omega)}W_{\tau(\omega)}(\omega)+\dfrac{t\wedge \tau(\omega)}{\tau(\omega)}Z(\omega),~ t \geq 0,~\omega\in \Omega.\label{eqdefofzeta}
		\end{equation}
	\end{definition}		
	The process $\zeta$ is obtained by composition of the maps $(r,t,z,\omega)\longmapsto \zeta^{r,z}_t(\omega)$ and $(t,\omega)\longmapsto (\tau(\omega),t,Z(\omega),\omega)$. According to Remark \ref{remarkmeasurable}, it is not hard to verify that the map $\zeta$ from $(\Omega,\mathbb{F})$ into
	$(C\left(\mathbb{R}_+,\mathbb{R}\right), \mathcal{C})$ is measurable.
	\begin{remark}
		Recall that the Brownian bridge with random length $\tau$ and deterministic pinning point $z$ associated to $W$, is defined by 
		\begin{equation}
		\zeta_t^{\tau,z}(\omega)=W_{t\wedge \tau(\omega)}(\omega)-\dfrac{t\wedge \tau(\omega)}{\tau(\omega)}W_{\tau(\omega)}(\omega)+\dfrac{t\wedge \tau(\omega)}{\tau(\omega)}z,~ t \geq 0,~\omega\in \Omega.\label{eqdefofzeta^z}
		\end{equation}
		Analogously, the Brownian bridge with deterministic length $r$ and random pinning point $Z$ associated to $W$, is defined by
		\begin{equation}
		\zeta_t^{r,Z}(\omega)=W_{t\wedge r}(\omega)-\dfrac{t\wedge r}{r}W_{r}(\omega)+\dfrac{t\wedge r}{r}Z(\omega),~ t \geq 0,~\omega\in \Omega.\label{eqdefofzeta^r}
		\end{equation}
		For a deeper discussion of the properties of these processes, we refer the reader to \cite{BBE}, \cite{EHL(Levy)}, \cite{EL}, \cite{EV} and \cite{HHM}.
	\end{remark}
	\begin{figure}[H]
		\centering
		\includegraphics[width=8cm]{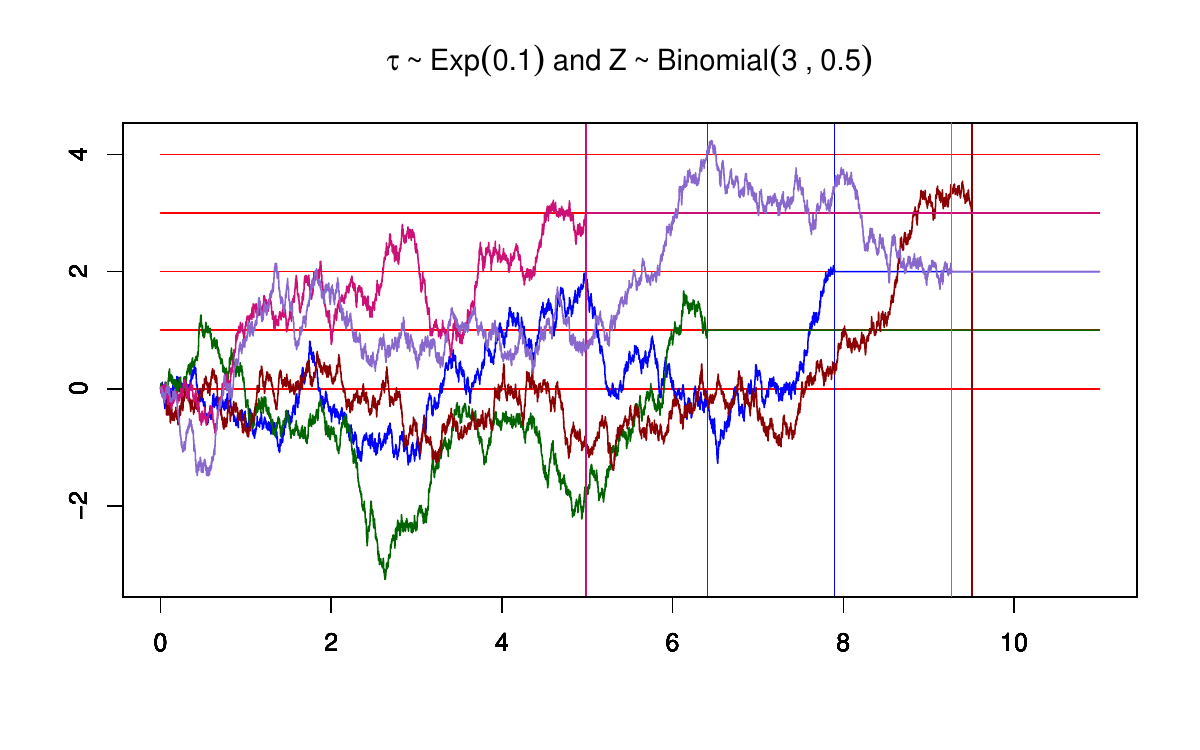}
		\includegraphics[width=8cm]{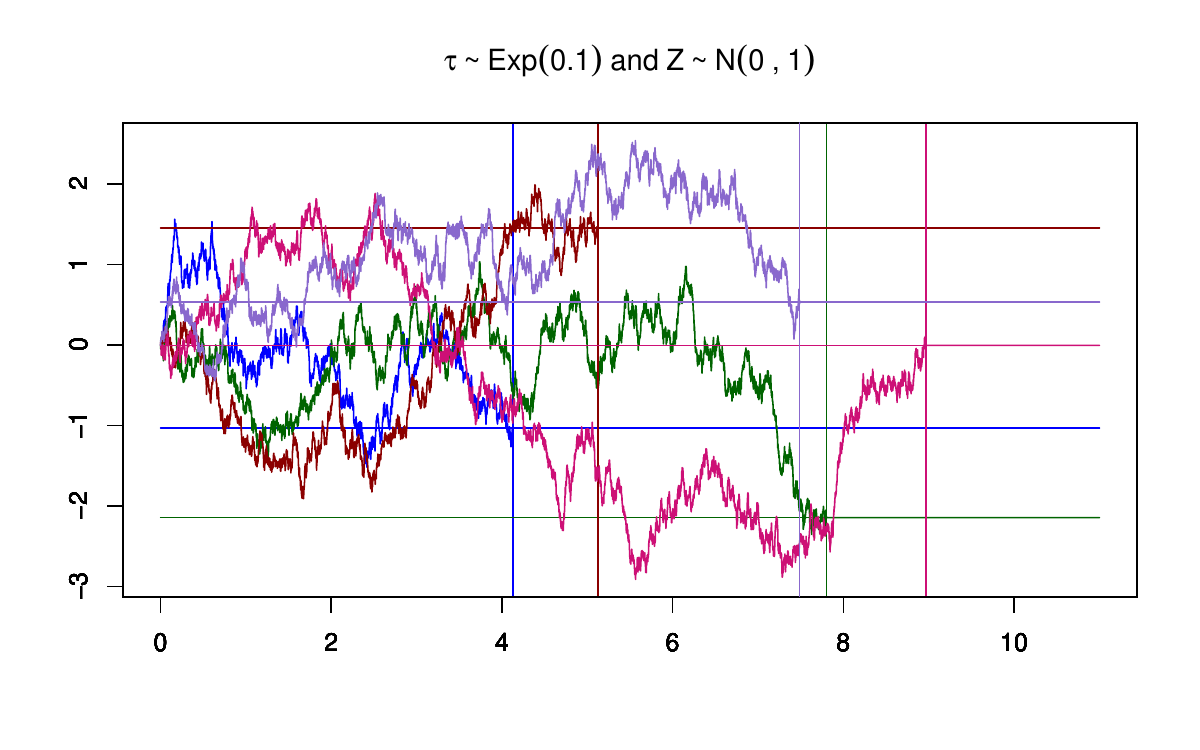}
		\caption{This figure represents simulated paths of a Brownian bridge with both random length and pinning point. In both pictures, the length follows an exponential distribution with rate parameter $\lambda=0.1$, whereas, the pinning point follows a binomial distribution with parameters $3$ and $0.5$ on the left-hand side, and it follows a standard normal distribution on the right-hand side.}
	\end{figure}
	The following natural assumption will be needed throughout the paper:
	\begin{hy}\label{hyindependenttauZW}
		For any Brownian bridge $\zeta$ with length $\tau$ and pinning point $Z$ associated to a Brownian motion $W$, we assume that $\tau$, $Z$ and $W$ are independent.
	\end{hy} 
	\begin{remark}
		Under the above assumption,	it is easy to see that, on the canonical space, for any Brownian bridge $\zeta$ with length $\tau$ and pinning point $Z$, we have
		\begin{equation}\label{condlawresptauZ}
		\mathbb{P}_{\zeta\vert \tau=r,Z=z}=\mathbb{P}_{\zeta^{r,z}},\, \mathbb{P}_{\zeta\vert Z=z}=\mathbb{P}_{\zeta^{\tau,z}}\, \text{and}\, \mathbb{P}_{\zeta\vert \tau=r}=\mathbb{P}_{\zeta^{r,Z}}
		\end{equation}
	\end{remark}
	If the random variable $Z$ is discrete, then we denote its state space by $\Delta=\{z_1,z_2,\ldots\}\subset \mathbb{R}$ and by $p_i$ the probability that the random variable $Z$ takes the value $z_i$. However, if the law of $Z$ is absolutely continuous with respect to the Lebesgue measure, we reserve the notation $f$ to represent its density.
	\begin{proposition}\label{propmeasurablity}
		Let $\zeta$ be the Brownian bridge with length $\tau$ and pinning point $Z$,	we have the following properties:
		\begin{enumerate}
			\item[(i)] The process $(\mathbb{I}_{\{\tau\leq t \}}, t> 0)$ is a modification of the process
			$(\mathbb{I}_{\{\zeta_t = Z \}}, t> 0)$  under the probability measure $\mathbb{P}$.
			\item[(ii)] If $Z$ is a discrete random variable, then for all $t>0$, the event $\{\tau\leq t \}\in \sigma(\zeta_t)\vee \mathcal{N}_p$.
			\item[(iii)] If the law of $Z$ is absolutely continuous with respect to the Lebesgue measure, then there exist times  $t>0$ such that $\{\tau\leq t \} \notin \sigma(\zeta_t)\vee \mathcal{N}_p$.
		\end{enumerate}
	\end{proposition}
	\begin{proof}
		\begin{enumerate}
			\item[(i)] It suffices to show that, for all $t>0$, we have $\mathbb{P}\left(\{\zeta_t = Z\} \bigtriangleup \{\tau \leq t\}\right)=0$. First we have from the definition of $\zeta$ that  $\zeta_t=Z$ for $\tau \leq t$. Then $\{\tau\leq t\}\subseteq \{\zeta_t=Z\}$. On the other hand, using the formula of total probability and \eqref{condlawresptauZ}, we obtain
			\begin{align*}
			\mathbb{P}(\zeta_{t}=Z,t<\tau)&=\dint_{\mathbb{R}}\dint_{0}^{+\infty} \mathbb{P}(\zeta_{t}=Z,t<\tau | \tau=r,Z=z) \mathbb{P}_{\tau}(\mathrm{d}r)\mathbb{P}_{Z}(\mathrm{d}z)
			\\  
			&=\dint_{\mathbb{R}}\dint_{t}^{+\infty}\mathbb{P}(\zeta_{t}^{r,z}=z) \mathbb{P}_{\tau}(\mathrm{d}r)\mathbb{P}_{Z}(\mathrm{d}z) \\
			&=0,
			\end{align*}
			where in the latter equality we have used only the the fact that, for all $0 < t < r$, the random variable $\zeta_t^{r,z}$ is absolutely continuous with respect to Lebesgue measure.
			\item[(ii)] We start with the observation that for any $t>0$, the event $\{\zeta_t=Z\}$ splits up disjointly into
			\begin{equation}
			\{ \zeta_t=Z \}=\overset{\infty}{\underset{i=1}{\bigcup}}\left( \{\zeta_t=z_i\}\bigcap \{Z=z_i\}\right),\label{eqzeta=Zinfunctionofzeta=zandZ=z}
			\end{equation}
			an equivalent formulation of \eqref{eqzeta=Zinfunctionofzeta=zandZ=z} is 
			\begin{equation}
			\mathbb{I}_{\{ \zeta_t=Z \}}=\overset{\infty}{\underset{i=1}{\sum}}\,\,\mathbb{I}_{\{ \zeta_t=z_i \}}\,\mathbb{I}_{\{ Z=z_i \}}.\label{eqzeta=Zinfunctionofzeta=zandZ=zindicator}
			\end{equation}
			On the other hand we have, for all $i$,
			\begin{align}
			\mathbb{E}[\mathbb{I}_{\{ \zeta_t=z_i \}}\mathbb{I}_{\{ Z\neq z_i \}}]&=\dint_{\mathbb{R}}\dint_{0}^{+\infty} \mathbb{P}(\zeta_{t}=z_i,Z \neq z_i | \tau=r,Z=z) \mathbb{P}_{\tau}(\mathrm{d}r)\mathbb{P}_{Z}(\mathrm{d}z)\nonumber\\
			&=\dint_{\mathbb{R}}\dint_{0}^{+\infty} \mathbb{P}(\zeta_{t}^{r,z}=z_i,z \neq z_i ) \mathbb{P}_{\tau}(\mathrm{d}r)\mathbb{P}_{Z}(\mathrm{d}z)\nonumber\\
			&=\overset{\infty}{\underset{j\neq i}{\sum}}\dint_{0}^{+\infty} \mathbb{P}(\zeta_{t}^{r,z_j}=z_i) \mathbb{P}_{\tau}(\mathrm{d}r)p_j\nonumber\\
			&=F_{\tau}(t)\overset{\infty}{\underset{j\neq i}{\sum}} \mathbb{I}_{\{z_j=z_i\}}p_j+\overset{\infty}{\underset{j\neq i}{\sum}}\dint_{t}^{+\infty} \mathbb{P}(\zeta_{t}^{r,z_j}=z_i) \mathbb{P}_{\tau}(\mathrm{d}r)p_j\nonumber\\
			&=0.\label{eqzeta=zincluZ=z}
			\end{align}
			Inserting \eqref{eqzeta=zincluZ=z} into the formula \eqref{eqzeta=Zinfunctionofzeta=zandZ=zindicator} we obtain, $\mathbb{P}$-a.s.,
			\begin{equation}
			\mathbb{I}_{\{ \zeta_t=Z \}}=\overset{\infty}{\underset{i=1}{\sum}}\,\,\mathbb{I}_{\{ \zeta_t=z_i \}},\label{eqzeta=Zinfunctionofzeta=z}
			\end{equation}
			using the fact that the process $(\mathbb{I}_{\{\tau\leq t \}}, t> 0)$ is a modification of the process
			$(\mathbb{I}_{\{\zeta_t = Z \}}, t> 0)$  under the probability measure $\mathbb{P}$, we have for all $t>0$, $\mathbb{P}$-a.s.,
			\begin{equation}
			\mathbb{I}_{\{ \tau \leq t \}}=\overset{\infty}{\underset{i=1}{\sum}}\,\,\mathbb{I}_{\{ \zeta_t=z_i \}},\label{eqtau<=t=zeta=Zinfunctionofzeta=z}
			\end{equation}
			which implies that for all $t>0$, we have $\{\tau \leq t\}\in\sigma(\zeta_t)\vee \mathcal{N}_p$.
			\item[(iii)] The desired assertion will be proved once we prove that there  exist times $t>0$ such that $\mathbb{E}[\mathbb{I}_{\{\tau \leq t\}}|\zeta_t]\neq \mathbb{I}_{\{\tau \leq t\}}$. In order to prove that we must determine the law of $\tau$ given $\zeta_t$. Due to \eqref{condlawresptauZ}, we have for $B \in \mathcal{B}(\mathbb{R})$ and $t>0$,
			\begin{align*}
			\mathbb{P}(\zeta_t\in B|\tau=r)&=\mathbb{P}(\zeta_t^{r,Z}\in B)\\
			&=\mathbb{P}(Z\in B)\mathbb{I}_{\{r\leq t\}}+\dint_{\mathbb{R}}\mathbb{P}(\zeta_t^{r,z}\in B)f(z)\mathrm{d}z\mathbb{I}_{\{t<r\}}\\
			&=\dint_{B}f(x)\mathrm{d}x\mathbb{I}_{\{r\leq t\}}+\dint_{\mathbb{R}}\dint_{B} \varphi_{\zeta_t^{r,z}}(x)\mathrm{d}xf(z)\mathrm{d}z\mathbb{I}_{\{t<r\}}\\
			&=\dint_{B}q_t(r,x)\mathrm{d}x,
			\end{align*}
			where $$q_t(r,x)=f(x)\mathbb{I}_{\{r\leq t\}}+\dint_{\mathbb{R}} \varphi_{\zeta_t^{r,z}}(x)f(z)\mathrm{d}z\mathbb{I}_{\{t<r\}}.$$
			From Remark \ref{remarkBayes}, it follows that for all bounded measurable functions $g$ we have, $\mathbb{P}$-a.s.,
			\begin{align}
			\mathbb{E}[g(\tau)|\zeta_t]&=\dfrac{\dint_{0}^{+\infty}g(r)q_t(r,\zeta_t)\mathbb{P}_{\tau}(\mathrm{d}r)}{\dint_{0}^{+\infty}q_t(r,\zeta_t)\mathbb{P}_{\tau}(\mathrm{d}r)}\nonumber\\
			&=\dfrac{f(\zeta_t)\dint_{0}^{t}g(r)\mathbb{P}_{\tau}(\mathrm{d}r)+\dint_{t}^{+\infty}g(r)\dint_{\mathbb{R}} \varphi_{\zeta_t^{r,z}}(\zeta_t)f(z)\mathrm{d}z\mathbb{P}_{\tau}(\mathrm{d}r)}{f(\zeta_t)F_{\tau}(t)+\dint_{t}^{+\infty}\dint_{\mathbb{R}} \varphi_{\zeta_t^{r,z}}(\zeta_t)f(z)\mathrm{d}z\mathbb{P}_{\tau}(\mathrm{d}r)}.\label{eqtaugivenzeta}
			\end{align}
			Consequently, we have, $\mathbb{P}$-a.s.,
			$$ \mathbb{E}[\mathbb{I}_{\{\tau \leq t\}}|\zeta_t]=\dfrac{f(\zeta_t)F_{\tau}(t)}{f(\zeta_t)F_{\tau}(t)+\dint_{t}^{+\infty}\dint_{\mathbb{R}} \varphi_{\zeta_t^{r,z}}(\zeta_t)f(z)\mathrm{d}z\mathbb{P}_{\tau}(\mathrm{d}r)},$$
			hence for $t>0$ such that $0<\mathbb{P}(\tau>t)<1$, we have, $\mathbb{P}$-a.s.,
			$$\dint_{t}^{+\infty}\dint_{\mathbb{R}} \varphi_{\zeta_t^{r,z}}(\zeta_t)f(z)\mathrm{d}z\mathbb{P}_{\tau}(\mathrm{d}r)\neq 0.$$
			Then there exist certain $t>0$, such that, $\mathbb{P}$-a.s.,
			$ \mathbb{E}[\mathbb{I}_{\{\tau \leq t\}}|\zeta_t]\neq \mathbb{I}_{\{\tau \leq t\}}$,
			which ends the proof.
		\end{enumerate}
	\end{proof}		
	\begin{proposition}\label{proplawoftauZgivenzeta_t}
		Let $t>0$ and $g$ be a measurable function on $(0, \infty) \times \mathbb{R}$
		such that $g(\tau,Z)$ is integrable. 
		\begin{enumerate}
			\item[(i)] If the law of the pinning point $Z$ is absolutely continuous with respect to the Lebesgue measure. Then, $\mathbb{P}$-a.s.,
			\begin{equation}
			\mathbb{E}[g(\tau,Z)|\zeta_t]=\dfrac{f(\zeta_t)\dint_{0}^{t}g(r,\zeta_t)\mathbb{P}_{\tau}(\mathrm{d}r)+\dint_{t}^{+\infty}\dint_{\mathbb{R}}g(r,z) \varphi_{\zeta_t^{r,z}}(\zeta_t)f(z)\mathrm{d}z\mathbb{P}_{\tau}(\mathrm{d}r)}{f(\zeta_t)F_{\tau}(t)+\dint_{t}^{+\infty}\dint_{\mathbb{R}} \varphi_{\zeta_t^{r,z}}(\zeta_t)f(z)\mathrm{d}z\mathbb{P}_{\tau}(\mathrm{d}r)}.\label{eqtauZgivenzeta_tcontinuous}
			\end{equation}
			\item[(ii)] If the pinning point $Z$ has a discrete distribution. Then, $\mathbb{P}$-a.s.,
			\begin{multline}
			\mathbb{E}[g(\tau,Z)|\zeta_t]=\underset{i\geq 1}{\sum}\dint_{0}^{t}\dfrac{g(r,z_i)}{F_{\tau}(t)}\;\mathbb{P}_{\tau}(\mathrm{d}r)\;\mathbb{I}_{\{\zeta_t=z_i\}}+\\\dint_{\mathbb{R}}\dint_{t}^{+\infty}g(r,z)\phi_{\zeta_{t}^{r,z}}(\zeta_t)\mathbb{P}_{\tau}(\mathrm{d}r)\mathbb{P}_{Z}(\mathrm{d}z)\;\mathbb{I}_{\{\zeta_t\neq Z\}},\label{eqtauZgivenzeta_tdiscrete}
			\end{multline}
			where $\varphi_{\zeta_{t}^{r,z}}$ is defined by \eqref{zeta^{r,z}density} and
			\begin{align}
			\phi_{\zeta_{t}^{r,z}}(x)
			&=\dfrac{\varphi_{\zeta_{t}^{r,z}}(x)}{\sum\limits _{i\geq 1}\dint_{t}^{+\infty}\varphi_{\zeta_{t}^{r,z_i}}(x)\mathbb{P}_{\tau}(\mathrm{d}r)p_i}\mathbb{I}_{\{ t<r\}}.
			\end{align}
		\end{enumerate}
	\end{proposition}
	\begin{proof}
		\begin{enumerate}
			\item[(i)] From \eqref{eqargumentconditioanal}, it follows that
			\begin{align}
			\mathbb{E}[g(\tau,Z)|\zeta_t=x]&=\dint_{0}^{+\infty}\mathbb{E}[g(\tau,Z)|\zeta_t=x,\tau=r]\,\mathbb{P}(\tau \in \mathrm{d}r|\zeta_t=x)\nonumber\\
			&=\dint_{0}^{+\infty}\mathbb{E}[g(r,Z)|\zeta_t^{r,Z}=x]\,\mathbb{P}(\tau \in \mathrm{d}r|\zeta_t=x).\label{eqtauZgivenzetaargument}
			\end{align}
			On the other hand, using the fact that $\zeta_t^{r,Z}=Z$ on $\{r \leq t\}$ and Bayes theorem, we have 
			\begin{equation}
			\mathbb{E}[g(r,Z)|\zeta_t^{r,Z}=x]=g(r,x)\mathbb{I}_{\{ r\leq t\}}+\dfrac{\dint_{\mathbb{R}}g(r,z)\varphi_{\zeta_{t}^{r,z}}(x)f(z)\mathrm{d}z}{\dint_{\mathbb{R}}\varphi_{\zeta_{t}^{r,z}}(x)f(z)\mathrm{d}z} \mathbb{I}_{\{ t<r\}}.\label{eqg(r,Z)givenzeta^r,Z}
			\end{equation}
			Hence, we obtain \eqref{eqtauZgivenzeta_tcontinuous} by substituting \eqref{eqtaugivenzeta} and \eqref{eqg(r,Z)givenzeta^r,Z} into \eqref{eqtauZgivenzetaargument}.
			\item[(ii)] In order to get this equality, it is convenient to use Bayes theorem. We have for all $t>0$,
			\begin{equation}
			\mathbb{P}(\zeta_t\in \mathrm{d}x,\tau\in \mathrm{d}r,Z\in \mathrm{d}z)=q_t(x,r,z)\mu(\mathrm{d}x)\mathbb{P}_{\tau}(\mathrm{d}r)\mathbb{P}_{Z}(\mathrm{d}z),\label{eqthelawofzetatauZdiscretecase}
			\end{equation}
			where
			\begin{equation}
			q_t(x,r,z)=\sum\limits _{i\geq 1}\left(\mathbb{I}_{\{x=z_i\}}\;\mathbb{I}_{\{z=z_i\}}\;\mathbb{I}_{\{r\leq t\}}\right)+\varphi_{\zeta_t^{r,z}}(x)\;\prod\limits _{i\geq 1}\mathbb{I}_{\{x\neq z_i\}}\mathbb{I}_{\{ t<r\}},\label{eqq_t(r,x,z)}
			\end{equation}
			and
			$$ \mu(\mathrm{d}x)=\mathrm{d}x+\sum\limits _{i\geq 1}\delta_{z_i}(\mathrm{d}x). $$	
			Indeed, Let $H$ be a bounded measurable function defined on $\mathbb{R}\times(0,\infty)\times\mathbb{R}$. For all $t>0$, we have 
			\begin{align*}
			\mathbb{E}[H(\zeta_t,\tau,Z)]&=\dint_{\mathbb{R}}\dint_{0}^{+\infty}\mathbb{E}[H(\zeta_{t},\tau,Z)\vert \tau=r, Z=z]\mathbb{P}_{\tau}(\mathrm{d}r)\mathbb{P}_{Z}(\mathrm{d}z)\nonumber\\
			&=\sum\limits _{i\geq 1}\left(\dint_{0}^{t}H(z_i,r,z_i)\mathbb{P}_{\tau}(\mathrm{d}r)\;+\dint_{t}^{+\infty}\mathbb{E}[H(\zeta_{t}^{r,z_i},r,z_i)]\mathbb{P}_{\tau}(\mathrm{d}r)\right)p_i\nonumber\\
			&=\sum\limits _{i\geq 1}\left(\dint_{0}^{t}H(z_i,r,z_i)\mathbb{P}_{\tau}(\mathrm{d}r)+\dint_{t}^{+\infty}\dint_{\mathbb{R}}H(x,r,z_i)\varphi_{\zeta_{t}^{r,z_i}}(x)\mathrm{d}x\mathbb{P}_{\tau}(\mathrm{d}r)\right)p_i\nonumber\\
			&=\dint_{\mathbb{R}}\dint_{0}^{+\infty}\dint_{\mathbb{R}}H(x,r,z)q_{t}(x,r,z)\mu(\mathrm{d}x)\mathbb{P}_{\tau}(\mathrm{d}r)\mathbb{P}_{Z}(\mathrm{d}z),
			\end{align*}
			using \eqref{eqthelawofzetatauZdiscretecase}, it is easy to see that the conditional law of $\zeta_t$ given $(\tau,Z)$ is given by
			\begin{equation}
			\mathbb{P}(\zeta_t\in \mathrm{d}x|\tau=r,Z=z)=q_{t}(x,r,z)\mu(\mathrm{d}x).
			\end{equation}
			Since the function $q_{t}$ is a non-negative and jointly measurable and $\mu$ is a $\sigma$-finite measure on $\mathbb{R}$, we conclude from Remark \ref{remarkBayes} that, $\mathbb{P}$-a.s.,
			\begin{align}
			\mathbb{E}[g(\tau,Z)\vert\zeta_{t}]=\dfrac{\dint_{\mathbb{R}}\dint_{0}^{+\infty}g(r,z)q_{t}(\zeta_t,r,z)\mathbb{P}_{\tau}(\mathrm{d}r)\mathbb{P}_{Z}(\mathrm{d}z)}{\dint_{\mathbb{R}}\dint_{0}^{+\infty}q_{t}(\zeta_t,r,z)\mathbb{P}_{\tau}(\mathrm{d}r)\mathbb{P}_{Z}(\mathrm{d}z)}.\label{eqbayestau,Zgivenzeta}
			\end{align}
			By a simple integration we obtain
			\begin{align*}
			\dint_{\mathbb{R}}\dint_{0}^{+\infty}g(r,z)q_{t}(\zeta_t,r,z)\mathbb{P}_{\tau}(\mathrm{d}r)\mathbb{P}_{Z}(\mathrm{d}z)&=\sum\limits _{i\geq 1}\dint_{0}^{t}g(r,z_i)\;\mathbb{P}_{\tau}(\mathrm{d}r)\;p_i\mathbb{I}_{\{\zeta_t=z_i\}}\\
			&+\dint_{\mathbb{R}}\dint_{t}^{+\infty}g(r,z)\varphi_{\zeta_{t}^{r,z}}(\zeta_t)\mathbb{P}_{\tau}(\mathrm{d}r)\mathbb{P}_{Z}(\mathrm{d}z)\prod\limits _{i\geq 1}\mathbb{I}_{\{\zeta_t\neq z_i\}}\\
			\end{align*}
			and
			\begin{align*}
			\dint_{\mathbb{R}}\dint_{0}^{+\infty}q_{t}(\zeta_t,r,z)\mathbb{P}_{\tau}(\mathrm{d}r)\mathbb{P}_{Z}(\mathrm{d}z)&=F_{\tau}(t)\sum\limits _{i\geq 1}\;p_i\mathbb{I}_{\{\zeta_t=z_i\}}\\
			&+\dint_{\mathbb{R}}\dint_{t}^{+\infty}\varphi_{\zeta_{t}^{r,z}}(\zeta_t)\mathbb{P}_{\tau}(\mathrm{d}r)\mathbb{P}_{Z}(\mathrm{d}z)\prod\limits _{i\geq 1}\mathbb{I}_{\{\zeta_t\neq z_i\}}.
			\end{align*}
			Combining all this leads to the following formula
			\begin{multline}
			\mathbb{E}[g(\tau,Z)\vert\zeta_{t}]=\sum\limits _{i\geq 1}\dint_{0}^{t}\dfrac{g(r,z_i)}{F_{\tau}(t)}\;\mathbb{P}_{\tau}(\mathrm{d}r)\;\mathbb{I}_{\{\zeta_t=z_i\}}\;\\+\dint_{\mathbb{R}}\dint_{t}^{+\infty}g(r,z)\phi_{\zeta_{t}^{r,z}}(\zeta_t)\mathbb{P}_{\tau}(\mathrm{d}r)\mathbb{P}_{Z}(\mathrm{d}z)\;\prod\limits _{i\geq 1}\mathbb{I}_{\{\zeta_t\neq z_i\}}
			\end{multline}
			we get the formula \eqref{eqtauZgivenzeta_tdiscrete} by using the fact that $\prod\limits _{i\geq 1}\mathbb{I}_{\{\zeta_t\neq z_i\}}=\mathbb{I}_{\{\zeta_t\neq Z\}}$, $\mathbb{P}$-a.s., which is an immediate consequence of \eqref{eqzeta=Zinfunctionofzeta=z}.
		\end{enumerate}
		
	\end{proof}
	\begin{corollary}\label{corlawoftauZzeta_ugivenzeta_t}
		Let $0<t<u$ and $g$ be a measurable function on $(0, \infty) \times \mathbb{R} \times \mathbb{R}$
		such that $g(\tau,Z,\zeta_u)$ is integrable.
		\begin{enumerate}
			\item[(i)] If the law of the pinning point $Z$ is absolutely continuous with respect to the Lebesgue measure. Then, $\mathbb{P}$-a.s.,
			\begin{equation}
			\mathbb{E}[g(\tau,Z,\zeta_u)|\zeta_t]=\dfrac{A_{t,u}(\zeta_t)}{f(\zeta_t)F_{\tau}(t)+\dint_{(t,\infty)}\dint_{\mathbb{R}} \varphi_{\zeta_t^{r,z}}(\zeta_t)f(z)dz\mathbb{P}_{\tau}(dr)},\label{eqlawoftauZzeta_ugivenzeta_tcontinuous}
			\end{equation}
			where
			\begin{multline}
			A_{t,u}(x)=f(x)\dint_{0}^{t}g(r,x,x)\mathbb{P}_{\tau}(\mathrm{d}r)+\dint_{t}^{u}\dint_{\mathbb{R}}g(r,z,z) \varphi_{\zeta_t^{r,z}}(x)f(z)\mathrm{d}z\mathbb{P}_{\tau}(\mathrm{d}r)\\
			+\dint_{u}^{+\infty}\dint_{\mathbb{R}}\dint_{\mathbb{R}}g(r,z,y)\varphi_{\zeta_u^{r,z},\zeta_t^{r,z}}(y,x)f(z)\mathrm{d}z\mathbb{P}_{\tau}(\mathrm{d}r).
			\end{multline}
			\item[(ii)] If the pinning point $Z$ has a discrete distribution. Then, $\mathbb{P}$-a.s.,
			\begin{multline}
			\mathbb{E}[g(\tau,Z,\zeta_u)|\zeta_t]=\underset{i\geq 1}{\sum}\dint_{0}^{t}\dfrac{g(r,z_i,z_i)}{F_{\tau}(t)}\;\mathbb{P}_{\tau}(\mathrm{d}r)\;\mathbb{I}_{\{\zeta_t=z_i\}}+\bigg[\dint_{\mathbb{R}}\dint_{t}^{u}g(r,z,z)\phi_{\zeta_{t}^{r,z}}(\zeta_t)\mathbb{P}_{\tau}(\mathrm{d}r)\mathbb{P}_{Z}(\mathrm{d}z)\\
			+\dint_{\mathbb{R}}\dint_{u}^{+\infty}\dint_{\mathbb{R}}g(r,z,y)[\mathbb{P}(\zeta_u^{r,z}\in \mathrm{d}y\vert \zeta_t^{r,z}=x)]_{x=\zeta_t}\,\phi_{\zeta_{t}^{r,z}}(\zeta_t)\mathbb{P}_{\tau}(\mathrm{d}r)\mathbb{P}_{Z}(\mathrm{d}z)\bigg]\mathbb{I}_{\{\zeta_t\neq Z\}}.\label{eqlawoftauZzeta_ugivenzeta_tdiscrete}
			\end{multline}
		\end{enumerate}
	\end{corollary}
	\begin{proof}
		The proof is an immediate consequence of Proposition \ref{proplawoftauZgivenzeta_t} and the following observation:
		\begin{align*}
		\mathbb{E}[g(\tau,Z,\zeta_u)|\zeta_t=x]&=\dint_{\mathbb{R}}\dint_{0}^{+\infty}\mathbb{E}[g(\tau,Z,\zeta_u)|\zeta_t=x,\tau=r,Z=z]\,\mathbb{P}(\tau \in \mathrm{d}r,Z\in \mathrm{d}z|\zeta_t=x)\\
		&=\dint_{\mathbb{R}}\dint_{0}^{+\infty}\mathbb{E}[g(r,z,\zeta_u^{r,z})|\zeta_t^{r,z}=x]\,\mathbb{P}(\tau \in \mathrm{d}r,Z\in \mathrm{d}z|\zeta_t=x).
		\end{align*}
	\end{proof}
	We have the following auxiliary result:
	\begin{lemma}\label{lemmamarkov}
		Assume that the law of $Z$ is absolutely continuous with respect to the Lebesgue measure and that there exists $t_1>0$ such that $F_{\tau}(t_1)=0$. Then for all $t_1<t_2<u$ and for every bounded measurable function $g$, we have, $\mathbb{P}$-a.s.,
		\begin{equation}
		\mathbb{E}[g(\zeta_u)|\zeta_{t_{1}},\zeta_{t_{2}}]=\dfrac{A_{t_1,t_2,u}(\zeta_{t_{1}},\zeta_{t_{2}})}{B_{t_1,t_2}(\zeta_{t_{1}},\zeta_{t_{2}})},\label{eqzeta_ugivent_1t_2}
		\end{equation}
		where
		\begin{align}
		A_{t_1,t_2,u}(\zeta_{t_{1}},\zeta_{t_{2}})&=g(\zeta_{t_2})f(\zeta_{t_2})\Psi_{t_1,t_2}(\zeta_{t_1},\zeta_{t_2})+
		p(t_2-t_1,\zeta_{t_{2}}-\zeta_{t_{1}})\bigg[\dint_{\mathbb{R}}g(z)\Psi_{t_2,u}(\zeta_{t_2},z)\,f(z)\mathrm{d}z\nonumber\\
		&+\dint_{\mathbb{R}}g(y)\dint_{u}^{+\infty}\Phi_{u}(r,y)\,\mathbb{P}_{\tau}(\mathrm{d}r)\,p(u-t_2,y-\zeta_{t_2})\,\mathrm{d}y\bigg],\label{eqA_{t_1,t_2,u}}
		\end{align}
		and
		\begin{equation}
		B_{t_1,t_2}(\zeta_{t_{1}},\zeta_{t_{2}})=f(\zeta_{t_2})\Psi_{t_1,t_2}(\zeta_{t_1},\zeta_{t_2})+p(t_2-t_1,\zeta_{t_{2}}-\zeta_{t_{1}})\dint_{t_2}^{+\infty}\Phi_{t_2}(r,\zeta_{t_2})\mathbb{P}_{\tau}(\mathrm{d}r).\label{eqB_{t_1,t_2}}
		\end{equation}
		Here, 
		\begin{equation}
		\Phi_{s}(t,y)=\dint_{\mathbb{R}}\frac{p(t-s,z-y)}{p(t,z)}\,f(z)\mathrm{d}z,\,s<t,\,\,y\in\mathbb{R},\label{eqPhi_s}
		\end{equation}	
		\begin{equation}
		\Psi_{s,t}(x,y)=\dint_{s}^{t}\frac{p(r-s,y-x)}{p(r,y)}\mathbb{P}_{\tau}(\mathrm{d}r),\,\,x,\,y\in\mathbb{R},\label{eqPsi_st}
		\end{equation}
		and
		\begin{equation}
		\Psi_{s}(x,y)=\dint_{s}^{+\infty}\frac{p(r-s,y-x)}{p(r,y)}\mathbb{P}_{\tau}(\mathrm{d}r),\,\,x,\,y\in\mathbb{R}.\label{eqPsi_s}
		\end{equation}
	\end{lemma}
	\begin{proof}
		The procedure is to find the law of $\zeta_u$ given $(\zeta_{t_1},\zeta_{t_2})$, for $0<t_1<t_2<u$ such that $F_{\tau}(t_1)=0$. For this purpose, let us first compute the law of $\tau$ given $(\zeta_{t_1},\zeta_{t_2})$. To do this, we need to use Bayes theorem. Let $B_1,B_2 \in \mathcal{B}(\mathbb{R})$, using the equality \eqref{condlawresptauZ}, the formula of total probability and the fact that $\mathbb{P}(\tau\leq t_1)=0$, we obtain
		\[
		\begin{array}{c}
		\mathbb{P}\left((\zeta_{t_1},\zeta_{t_2})\in B_1 \times B_2\vert \tau=r\right)=\mathbb{P}\left((\zeta_{t_{1}}^{r},\zeta_{t_{2}}^{r})\in B_{1}\times B_{2}\right)=\dint_{\mathbb{R}}\mathbb{P}\left((\zeta_{t_{1}}^{r,z},\zeta_{t_{2}}^{r,z})\in B_{1}\times B_{2}\right)\mathbb{P}_{Z}(\mathrm{d}z)
		\\ \\ =\dint_{B_{1}\times B_{2}}q_{t_{1},t_{2}}(r,x_{1},x_{2})\mathrm{d}x_1\,\mathrm{d}x_2.
		
		\end{array}
		\]
		where
		\begin{align*}
		q_{t_1,t_2}(r,x_1,x_2)&=\varphi_{\zeta_{t_1}^{r,x_2}}(x_1)f(x_2)\mathbb{I}_{\{t_1<r\leq t_2\}}+\dint_{\mathbb{R}}\varphi_{\zeta_{t_1}^{r,z},\zeta_{t_2}^{r,z}}(x_1,x_2)f(z)\mathrm{d}z\mathbb{I}_{\{t_2<r\}}\\
		&=p(t_1,x_1)\bigg[\dfrac{p(r-t_1,x_2-x_1)}{p(r,x_2)}f(x_2)\mathbb{I}_{\{t_1<r\leq t_2\}}+\Phi_{t_2}(r,x_2)p(t_2-t_1,x_2-x_1)\mathbb{I}_{\{t_2<r\}}\bigg].
		\end{align*}
		Then from Remark \ref{remarkBayes}, we have for all $(x_1,x_2)\in \mathbb{R}^2$,
		\begin{align}
		&\mathbb{E}[g(\tau)|\zeta_{t_1}=x_1,\zeta_{t_2}=x_2]=\dfrac{\dint_{t_1}^{+\infty}g(r)q_{t_1,t_2}(r,x_1,x_2)\mathbb{P}_{\tau}(\mathrm{d}r)}{\dint_{t_1}^{+\infty}q_{t_1,t_2}(r,x_1,x_2)\mathbb{P}_{\tau}(\mathrm{d}r)}\nonumber\\
		&=\dfrac{f(x_2)\dint_{t_1}^{t_2}g(r)\dfrac{p(r-t_1,x_2-x_1)}{p(r,x_2)}\mathbb{P}_{\tau}(\mathrm{d}r)+p(t_2-t_1,x_2-x_1)\dint_{t_2}^{+\infty}g(r)\Phi_{t_2}(r,x_2)\mathbb{P}_{\tau}(\mathrm{d}r)}{f(x_2)\Psi_{t_1,t_2}(x_1,x_2)+p(t_2-t_1,x_2-x_1)\dint_{t_2}^{+\infty}\Phi_{t_2}(r,x_2)\mathbb{P}_{\tau}(\mathrm{d}r)}.\label{eqtaugivenzetat-1,t-2}
		\end{align}
		Applying \eqref{eqargumentconditioanal}, we obtain
		\begin{align}
		\mathbb{E}[g(\zeta_{u})\vert\zeta_{t_{1}}=x_1,\zeta_{t_{2}}=x_2]&=\dint_{(t_1,\infty)} \mathbb{E}[g(\zeta_{u})\vert\zeta_{t_{1}}=x_1,\zeta_{t_{2}}=x_2,\tau=r]\mathbb{P}(\tau\in \mathrm{d}r|\zeta_{t_{1}}=x_1,\zeta_{t_{2}}=x_2)\nonumber\\
		&=\dint_{(t_1,\infty)} \mathbb{E}[g(\zeta_{u}^{r,Z})\vert\zeta_{t_{1}}^{r,Z}=x_1,\zeta_{t_{2}}^{r,Z}=x_2]\mathbb{P}(\tau\in \mathrm{d}r|\zeta_{t_{1}}=x_1,\zeta_{t_{2}}=x_2)\nonumber\\
		&=\dint_{(t_1,\infty)} \mathbb{E}[g(\zeta_{u}^{r,Z})\vert\zeta_{t_{2}}^{r,Z}=x_2]\mathbb{P}(\tau\in \mathrm{d}r|\zeta_{t_{1}}=x_1,\zeta_{t_{2}}=x_2).\label{eqzeta_ugivenzetat-1,t-2}
		\end{align}
		The latter equality uses the fact that $\zeta_t^{r,Z}$ is a Markov process with respect to its natural filtration, see Proposition 3.4 in  \cite{HHM}. On the other hand it is clear that
		\begin{equation*}
		\mathbb{E}[g(\zeta_{u}^{r,Z})\vert\zeta_{t_{2}}^{r,Z}=x_2]=g(x_2)\mathbb{I}_{\{t_1<r\leq t_2\}}+\mathbb{E}[g(Z)|\zeta_{t_2}^{r,Z}=x_2]\mathbb{I}_{\{t_2<r\leq u\}}+\mathbb{E}[g(\zeta_u^{r,Z})|\zeta_{t_2}^{r,Z}=x_2]\mathbb{I}_{\{u<r\}}.
		\end{equation*}
		Using Proposition 3.4 in \cite{HHM}, we obtain that
		\begin{align}
		\mathbb{E}[g(\zeta_{u}^{r,Z})\vert\zeta_{t_{2}}^{r,Z}=x_2]&=g(x_2)\mathbb{I}_{\{t_1<r\leq t_2\}}+\dfrac{\dint_{\mathbb{R}}g(z)\frac{p(r-t_2,z-x_2)}{p(r,z)}f(z)\mathrm{d}z}{\Phi_{t_2}(r,x_2)}\,\mathbb{I}_{\{t_2<r\leq u\}}\nonumber\\
		&+\dint_{\mathbb{R}}g(y)\dfrac{\Phi_{u}(r,y)}{\Phi_{t_2}(r,x_2)}p(u-t_2,y-x_2)\mathrm{d}y\mathbb{I}_{\{u<r\}}.\label{eqsplit}
		\end{align}
		We obtain \eqref{eqzeta_ugivent_1t_2} by inserting \eqref{eqsplit} into the formula  \eqref{eqzeta_ugivenzetat-1,t-2} and applying \eqref{eqtaugivenzetat-1,t-2}.
	\end{proof}
	\begin{remark}
		In order to determine the law of $\tau$ given $(\zeta_{t_1},\zeta_{t_2})$, it is necessary to assume that the length $\tau$ satisfies $\mathbb{P}(\tau \leq t_1)=0$, otherwise we cannot use Bayes theorem. Indeed,  we cannot find, both; a function $q_{t_1,t_2}(.,.,.)$, which is non-negative and jointly measurable with respect to the three variables, and a $\sigma$-finite measure $\mu$ such that for every $B_1,B_2 \in \mathcal{B}(\mathbb{R})$,
		\[
		\begin{array}{c}
		\mathbb{P}\left((\zeta_{t_1},\zeta_{t_2})\in B_1 \times B_2\vert \tau=r\right)=\dint_{B_{1}\times B_{2}}q_{t_{1},t_{2}}(r,x_{1},x_{2})\mu(\mathrm{d}x_1\,\mathrm{d}x_2).
		\end{array}
		\]
		Due to the fact that for random variables $Z$ having a law absolutely continuous with respect to the Lebesgue measure on $\mathbb{R}$, the law of the couple $(Z,Z)$ is not absolutely continuous with respect to the Lebesgue measure on $\mathbb{R}^2$.
	\end{remark}
	We are now in a position to discuss the Markov property of the Brownian bridge with random length and pinning point. We first consider pinning points $Z$ with absolutely continuous law with respect to the Lebesgue measure. Below we show that there exist random times $\tau$ such that the Brownian bridge $\zeta$ with length $\tau$ and pinning point $Z$ is not an $\mathbb{F}^{\zeta,c}$-Markov process.
	\begin{theorem}
		Let $Z$ be a random variable having a law which is absolutely continuous with respect to the Lebesgue measure. There exist random times $\tau$ such that the Brownian bridge $\zeta$ with length $\tau$ and pinning point $Z$ does not possess the Markov property with respect to its natural filtration.
	\end{theorem}
	\begin{proof}
		Let $Z$ have a law, which is absolutely continuous with respect to the Lebesgue measure and $\tau$ be a two-point random variable such that $\mathbb{P}(\tau=T_1)=\mathbb{P}(\tau=T_2)=\frac{1}{2}$.
		Our goal is to prove that the Brownian bridge $\zeta$ with length $\tau$ and a pinning point $Z$ is not a Markov process with respect to its natural filtration.
		Using Theorem 1.3 in Blumenthal and Getoor \cite{BG}, it is sufficient to prove that there exist a bounded measurable function $g$ and $0<t_1<t_2<u$ such that
		\begin{equation}
		\mathbb{E}[g(\zeta_u)|\zeta_{t_1},\zeta_{t_2}]\neq\mathbb{E}[g(\zeta_u)|\zeta_{t_2}].\label{eqzetaisnotMarkovian}
		\end{equation}
		Let $t_1$, $t_2$ and $u$ such that $0<t_1<T_1<t_2<u<T_2$, then $F_{\tau}(t_1)=0$ and $F_{\tau}(t_2)=\frac{1}{2}$. Moreover, for every $x, y\in \mathbb{R}$, we have
		\begin{equation*}
		\Psi_{t_2,u}(x,y)=\dint_{t_2}^{u}\frac{p(r-t_2,y-x)}{p(y,x)}\mathbb{P}_{\tau}(\mathrm{d}r)=0.
		\end{equation*}
		It follows from Lemma \ref{lemmamarkov} that for every bounded measurable function $g$, we have, $\mathbb{P}$-a.s.,
		\begin{align}
		\mathbb{E}[g(\zeta_u)|\zeta_{t_1},\zeta_{t_2}]&=\dfrac{g(\zeta_{t_2})f(\zeta_{t_2})\frac{p(T_1-t_1,\zeta_{t_2}-\zeta_{t_1})}{p(T_1,\zeta_{t_2})}+p(t_2-t_1,\zeta_{t_2}-\zeta_{t_1})\dint_{\mathbb{R}}g(y)\Phi_{u}(T_2,y)p(u-t_2,y-\zeta_{t_2})\mathrm{d}y}{f(\zeta_{t_2})\frac{p(T_1-t_1,\zeta_{t_2}-\zeta_{t_1})}{p(T_1,\zeta_{t_2})}+p(t_2-t_1,\zeta_{t_2}-\zeta_{t_1})\Phi_{t_2}(T_2,\zeta_{t_2})}\nonumber\\
		&=\dfrac{g(\zeta_{t_2})f(\zeta_{t_2})+H_{t_1,T_1,t_2}(\zeta_{t_1},\zeta_{t_2})\dint_{\mathbb{R}}g(y)\Phi_{u}(T_2,y)p(u-t_2,y-\zeta_{t_2})\mathrm{d}y}{f(\zeta_{t_2})+H_{t_1,T_1,t_2}(\zeta_{t_1},\zeta_{t_2})\Phi_{t_2}(T_2,\zeta_{t_2})},\label{eqnotMarkov1}
		\end{align}
		where
		\begin{equation*}
		H_{t_1,T_1,t_2}(x,y)=\dfrac{p(t_2-t_1,y-x)}{p(T_1-t_1,y-x)}p(T_1,y),\,\,x,y\in\mathbb{R}.
		\end{equation*}
		On the other hand, \eqref{eqlawoftauZzeta_ugivenzeta_tcontinuous} implies that, $\mathbb{P}$-a.s.,
		\begin{align}
		\mathbb{E}[g(\zeta_u)|\zeta_{t_2}]=\dfrac{g(\zeta_{t_2})f(\zeta_{t_2})+p(t_2,\zeta_{t_2})\dint_{\mathbb{R}}g(y)\Phi_{u,T_2}(y)\,p(u-t_2,y-\zeta_{t_2})\mathrm{d}y}{f(\zeta_{t_2})+p(t_2,\zeta_{t_2})\Phi_{t_2,T_2}(\zeta_{t_2})}.\label{eqnotMarkov2}
		\end{align}
		Since
		\begin{equation}
		\dfrac{p(t_2-t_1,\zeta_{t_2}-\zeta_{t_1})}{p(T_1-t_1,\zeta_{t_2}-\zeta_{t_1})}\neq \dfrac{p(t_2,\zeta_{t_2})}{p(T_1,\zeta_{t_2})},
		\end{equation}
		it follows from \eqref{eqnotMarkov1} and \eqref{eqnotMarkov2} that 
		$$ \mathbb{E}[g(\zeta_u)|\zeta_{t_1},\zeta_{t_2}]\neq\mathbb{E}[g(\zeta_u)|\zeta_{t_2}].$$
		This is precisely the assertion of the theorem.
	\end{proof}
	In the following result, we show that if $Z$ is a random variable having a discrete distribution then for any random time $\tau$, the Brownian bridge $\zeta$ with length $\tau$ and pinning point $Z$ is a Markov process with respect to its natural filtration $\mathbb{F}^{\zeta}$.
	\begin{theorem}\label{thmMarkovdiscrete}
		Assume that $Z$ is a random variable having a discrete distribution. Then for any random time $\tau$, the Brownian bridge $\zeta$ with length $\tau$ and pinning point $Z$ is an $\mathbb{F}^{\zeta}$-Markov process.
	\end{theorem}
	\begin{proof}
		First, since $\zeta_{0}=0$ almost surely, it is sufficient to prove that for each finite collection $0<t_0<t_1\leq \cdots \leq t_n=t<u$ and for every  bounded measurable function $f$ one has
		\begin{equation}
		\mathbb{E}[f(\zeta_{u})\vert \zeta_{t_n},\ldots,\zeta_{t_0}]=\mathbb{E}[f(\zeta_{u})\vert\zeta_{t_n}].\label{eqMarkovdiscrete}
		\end{equation}
		Using property (ii) of Proposition \ref{propmeasurablity} together with the fact that the pinning point $Z$ is discrete we conclude that, for all $t>0$, $\{\tau\leq t \}\in \sigma(\zeta_t)\vee \mathcal{N}_P$.  Since $f(\zeta_{u})\mathbb{I}_{\{\tau\leq t\}}=f(\zeta_t)\mathbb{I}_{\{\tau\leq t\}}$ which is measurable with respect to $\sigma(\zeta_t)\vee \mathcal{N}_{P}$, it remains to prove \eqref{eqMarkovdiscrete} on the set $\{t<\tau \}$, that is, $\mathbb{P}\text{-a.s.},$
		\begin{equation}
		\mathbb{E}[f(\zeta_{u})\mathbb{I}_{\{t<\tau \}}\vert \zeta_{t_n},\ldots,\zeta_{t_0}]=\mathbb{E}[f(\zeta_{u})\mathbb{I}_{\{t<\tau \}}\vert\zeta_{t_n}].\label{eqMarkovdiscretet<tau}
		\end{equation}
		Setting,
		$$ \gamma_k=\dfrac{\zeta_{t_k}}{t_k}-\dfrac{\zeta_{t_{k-1}}}{t_{k-1}}\,\,\, \text{and}\,\,\, \alpha_k=\dfrac{W_{t_k}}{t_k}-\dfrac{W_{t_{k-1}}}{t_{k-1}},\,k=1,\cdots,n. $$
		By using the fact that $\gamma_k=\alpha_k$ on the set $\{t<\tau\}$ and that
		\begin{equation}
		\mathbb{E}[f(\zeta_{u})\mathbb{I}_{\{t<\tau\}}\vert \zeta_{t_n},\cdots,\zeta_{t_0}]=\mathbb{E}[f(\zeta_{u})\mathbb{I}_{\{t<\tau\}}\vert \zeta_{t_n},\gamma_n,\cdots,\gamma_{1}],
		\end{equation}
		we are reduced to proving that for any bounded measurable functions $g$ and $h$ on $\mathbb{R}^n$ and $\mathbb{R}$, respectively, we have
		\begin{equation}
		\mathbb{E}[f(\zeta_{u})\mathbb{I}_{\{t<\tau\}}h(\zeta_{t_n})g(\alpha_n,\cdots,\alpha_{1})]=\mathbb{E}[\mathbb{E}[f(\zeta_{u})\vert \zeta_{t_n} ]\mathbb{I}_{\{t<\tau\}}h(\zeta_{t_n})g(\alpha_n,\cdots,\alpha_{1})].\label{equationAcapt<tau}
		\end{equation}
		On the other hand, for $t<r$ and $z\in \mathbb{R}$, it is easy to see that the vectors $(\alpha_1,\cdots,\alpha_n)$ and $(\zeta_{t}^{r,z},\zeta_{u}^{r,z})$ are independent, which implies by using the fact that the vector $(\alpha_1,\cdots,\alpha_n)$ does not depend on $\tau$ and $Z$, the formula of total probability, and \eqref{condlawresptauZ}, that the random variables $f(\zeta_{u})h(\zeta_t)\mathbb{I}_{\{t<\tau \}}$ and $g(\alpha_1,\ldots,\alpha_n)$ are uncorrelated. Now taking into account all the above considerations, we have 
		\begin{align*}
		\mathbb{E}[f(\zeta_{u})\mathbb{I}_{\{t<\tau\}}h(\zeta_{t_n})g(\alpha_n,\cdots,\alpha_{1})]&=\mathbb{E}[f(\zeta_{u})h(\zeta_{t_n})\mathbb{I}_{\{t<\tau\}}]\mathbb{E}[g(\alpha_n,\cdots,\alpha_{1})]\\
		\\
		& =\mathbb{E}[\mathbb{E}[f(\zeta_{u})\vert \zeta_{t_n} ]\mathbb{I}_{\{t<\tau\}}h(\zeta_{t_n})]\mathbb{E}[g(\alpha_n,\cdots,\alpha_{1})]\\
		\\
		& =\mathbb{E}[\mathbb{E}[f(\zeta_{u})\vert \zeta_{t_n} ]\mathbb{I}_{\{t<\tau\}}h(\zeta_{t_n})g(\alpha_n,\cdots,\alpha_{1})].
		\end{align*}
		Hence \eqref{equationAcapt<tau} is proved.
	\end{proof}
	\begin{remark}
		It is not hard to see that the Markov property can be extended to the completed filtration $\mathbb{F}^{\zeta,c}$.
	\end{remark}
	\begin{center}
		\section{Brownian Bridge Information Process}
	\end{center}
	Developed countries are increasingly relying on gas storage to ensure security of supply. As a consequence gas storage is traded. The value of storage is derived from the possibility of buying and injecting gas at times of low prices and withdrawing and selling when prices are high. Given past price behaviour and a predictive model for future prices, recent research has examined the problem of how to manage the injection/withdrawal schedule of gas, faced by owners of storage contracts, by applying real option theory. Many papers are devoted to tackle this problem, Holland \cite{H} proposes Monte Carlo Simulation while Boogert and de Jong \cite{BJ} and Carmona and Ludkovski \cite{CL} apply Least Squares Monte Carlo Simulations. In contrast to the traditional approaches adopted in the literature, our approach aims to specify a model for the market filtration. We explicitly model the market filtration as being generated by a Brownian bridge and providing a noisy, partial flow of information about the future decision to inject respectively, withdraw gas. For a deeper discussion of such an approach we refer the reader to \cite{BHM2007}, \cite{BHM2008} and \cite{BBE}. Our suggested approach is motivated by Bedini et al. \cite{BBE}, who address credit risk. The time instant at which default occurs is represented by a strictly positive random time $\tau$ and the flow of information on future default is modelled by the completed natural filtration $\mathbb{F}^{\beta,c}$ associated with the Brownian bridge $\beta$ with random length $\tau$, i.e.
	$$\beta_t=W_{t\wedge \tau}-\dfrac{t\wedge \tau}{\tau}W_{\tau},\,\,t\geq 0,$$ 	
	where $W$ is a Brownian motion independent of $\tau$. 
	We adopt and generalize the model for the decision of the holder of a gas storage contract whether to inject or withdraw gas at pre-set levels $z_1<z_2$ at a random positive action time $\tau$, while staying inactive before $\tau$. The flow of information that motivates the holder of a gas contract to remain inactive before $\tau$ and to make an action at time $\tau$ is modelled by the completed natural filtration generated by a Brownian bridge process $\xi=(\xi_t, t\geq 0)$ starting from zero and conditioned to be equal to a constant $z_1$ at the time of injection and a constant $z_2$ at the time of withdrawal. It follows that the process $\xi$ takes the form
	\begin{equation}
	\xi_t(\omega):=W_{t\wedge \tau(\omega)}(\omega)-\dfrac{t\wedge \tau(\omega)}{\tau(\omega)}W_{\tau(\omega)}(\omega)+\dfrac{t\wedge \tau(\omega)}{\tau(\omega)}Z(\omega),~ t \geq 0,~\omega\in \Omega,\label{eqdefofxi}
	\end{equation} 
	where $$Z= \left\{
	\begin{array}{lll}
	& z_1\,\, \text{with probability}\,\,p_1,  
	
	\\ \\ & z_2\,\, \text{with probability}\,\,p_2=1-p_1.
\end{array}
\right.$$
Our approach aims to give a description of the information on the time action $\tau$. Since the information will be carried by $\xi=(\xi_t, t\geq 0)$, we call $\xi$ the Brownian bridge information process. The filtration $\mathbb{F}^{\xi}$ generated by the information process provides partial information on $\tau$. The intuitive idea is that away from the boundaries $z_1<z_2$ the Brownian bridge information process models the holder's motivation for remaining inactive. Alternatively, the case when the Brownian absorbs at $z_1$ (resp. absorbs at $z_2$) models the decision of injecting gas (resp. withdrawing gas). Since the holder of the contract can withdraw and inject gas multiple times, we reset the process at each time the holder makes an action. In this sense $\xi$ leaks information concerning the time $\tau$ at which the holder of a storage contract chooses to inject gas, do nothing, or withdraw gas. In order to compile some facts on $\tau$ it is important to study the properties of the process $\xi$. Specifically, the Markov property, the right continuity of its natural filtration and its semi-martingale decomposition. We emphasize that the issue of storage valuation is not limited to gas markets, it also plays a significant, balancing role in, for example, oil markets, soft commodity markets and even electricity. The principle of our approach is applicable to those markets as well.
We recall that we are working under the assumption that, the random time $\tau$, the pinning point $Z$ and the Brownian motion are independent. Since in this case the pinning point $Z$ follows a discrete distribution, the following result is a consequence of Theorem \ref{thmMarkovdiscrete} and Corollary \ref{corlawoftauZzeta_ugivenzeta_t}.
\begin{theorem}
The Brownian bridge information process is an $\mathbb{F}^{\xi}$-Markov process, with transition densities given by:
\begin{align}
\mathbb{P}(\xi_u\in \mathrm{d}y\vert \xi_t&=x)=\bigg[  \sum\limits _{i=1}^{2}\bigg(\mathbb{I}_{\{x=z_i\}}+\dfrac{\Psi_{t,u}(x,z_i)p_i}{\Psi_{t}(x,z_1)p_1+\Psi_{t}(x,z_2)p_2}\,\mathbb{I}_{\{x\neq z_1,x\neq z_2\}}\bigg) \mathbb{I}_{\{y=z_i\}}
\nonumber\\
&+p(u-t,y-x)\,\dfrac{\Psi_{u}(y,z_1)p_1+\Psi_{u}(y,z_2)p_2}{\Psi_{t}(x,z_1)p_1+\Psi_{t}(x,z_2)p_2}\,\mathbb{I}_{\{x\neq z_1,x\neq z_2\}}\mathbb{I}_{\{y\neq z_1,y\neq z_2\}}\bigg]\mu(\mathrm{d}y)
\end{align}
for all $0<t<u$. Where
$$\mu(\mathrm{d}y)=\delta_{z_1}(\mathrm{d}y)+\delta_{z_2}(\mathrm{d}y)+\mathrm{d}y.$$
\end{theorem}
\begin{remark}
It is clear that the transition density $ \mathbb{P}(\xi_u\in \mathrm{d}y|\xi_t=x)$ fails to depend only on $u - t$, hence, the Brownian bridge information process $\xi$ cannot be an homogeneous $\mathbb{F}^{\xi}$-Markov process.
\end{remark}
In the following result we state the stopping time property of the random action time $\tau$. Furthermore, we describe the structure of the a posteriori distribution of $(\tau,Z)$ based on the observation of the Brownian bridge information process $\xi$ up to time $t$.
\begin{theorem}\label{theorembayes}
The random action time $\tau$ is an $\mathbb{F}^{\xi,c}$-stopping time. Moreover, for any $t>0$ and for every measurable function on $(0,\infty)\times \mathbb{R}$ such that $g(\tau,Z)$
is integrable, we have, $\mathbb{P}$-a.s.,
\begin{equation}
\mathbb{E}[g(\tau,Z)|\mathcal{F}_t^{\xi,c}]=g(\tau,Z)\mathbb{I}_{\{\xi_t=Z\}}+\dfrac{\sum\limits _{i=1}^{2}\dint_{t}^{+\infty}g(r,z_i)\dfrac{p(r-t,z_i-\xi_t)}{p(r,z_i)}\mathbb{P}_{\tau}(\mathrm{d}r)p_i}{\Psi_{t}(\xi_t,z_1)p_1+\Psi_{t}(\xi_t,z_2)p_2}\,\mathbb{I}_{\{\xi_t\neq Z\}}.\label{eqlawoftauZgivenfzeta}
\end{equation}
\end{theorem}
\begin{proof}
Since the pinning point $Z$ is a random variable having a discrete distribution, it follows from the second assertion of Proposition \ref{propmeasurablity} that the random time $\tau$ is an $\mathbb{F}^{\xi,c}$-stopping time.
It remains to prove \eqref{eqlawoftauZgivenfzeta}. Obviously, we have
$$\mathbb{E}[g(\tau,Z)\vert \mathcal{F}_{t}^{\xi,c}]=\mathbb{E}[g(\tau,Z)\mathbb{I}_{\{\tau\leqslant t\}}\vert \mathcal{F}_{t}^{\xi,c}]+\mathbb{E}[g(\tau,Z)\mathbb{I}_{\{ t<\tau \}}\vert \mathcal{F}_{t}^{\xi,c}].$$
Since $\tau$ is an $\mathbb{F}^{\xi,c}$-stopping time and $Z$ is $\mathcal{F}_{\tau}^{\xi,c}$-measurable, it follows that $g(\tau,Z)\mathbb{I}_{\{\tau\leqslant t\}}$ is $ \mathcal{F}^{\xi,c}_{t}$-measurable then, $\mathbb{P}$-a.s, one has
\begin{align*}
\mathbb{E}[g(\tau,Z)\mathbb{I}_{\{\tau\leqslant t\}}\vert\mathcal{F}_{t}^{\xi,c}]&=g(\tau,Z)\mathbb{I}_{\{\tau\leqslant t\}}
\\
&=g(\tau,Z)\mathbb{I}_{\{\xi_{t}=Z\}}.
\end{align*}
On the other hand, due to the fact that $g(\tau,Z)\mathbb{I}_{\{t< \tau \}}$ is  $\sigma(\xi_s, t \leq s \leq +\infty) \vee \mathcal{N}_P$-measurable and $\xi$ is a Markov process with respect to its completed natural filtration we obtain, $\mathbb{P}$-a.s.,
$$\mathbb{E}[g(\tau,Z)\mathbb{I}_{\{ t<\tau \}}\vert \mathcal{F}_{t}^{\xi,c}]=\mathbb{E}[g(\tau,Z)\mathbb{I}_{\{ t<\tau \}}|\xi_{t}].$$
The result is deduced from \eqref{eqtauZgivenzeta_tdiscrete}.
\end{proof}
The following result extends Theorem \ref{theorembayes}.
\begin{corollary}
Let $0<t<u$ and $g$ be a measurable function on $(0, \infty) \times \mathbb{R} \times \mathbb{R}$
such that $g(\tau,Z,\xi_u)$ is integrable. Then, $\mathbb{P}$-a.s.,
\begin{multline}
\mathbb{E}[g(\tau,Z,\xi_u)|\mathcal{F}_t^{\xi,c}]=g(\tau,Z,Z)\mathbb{I}_{\{\xi_t=Z\}}+\bigg[\sum\limits _{i=1}^{2}\dint_{t}^{u}g(r,z_i,z_i)\phi_{\zeta_{t}^{r,z_i}}(\xi_t)\mathbb{P}_{\tau}(\mathrm{d}r)p_i\\
+\sum\limits _{i=1}^{2}\dint_{u}^{+\infty}\dint_{\mathbb{R}}g(r,z_i,y)[\mathbb{P}(\zeta_u^{r,z_i}\in \mathrm{d}y\vert \zeta_t^{r,z_i}=x)]_{x=\xi_t}\,\phi_{\zeta_{t}^{r,z_i}}(\xi_t)\mathbb{P}_{\tau}(\mathrm{d} r)p_i\bigg]\mathbb{I}_{\{\xi_t\neq Z\}}.\label{eqlawoftauZgivenxi_t}
\end{multline}
\end{corollary}
\begin{proof}
Due to the fact that $\xi$ is an $\mathbb{F}^{\xi,c}$-Markov process, $g(\tau,Z,\xi_u)\mathbb{I}_{\{ u<\tau \}}$ is measurable with respect to $\sigma(\xi_s, u \leq s \leq +\infty) \vee \mathcal{N}_P$, and that $\xi_u=Z$ on the set $\{ \tau \leq u \}$, we obtain, $\mathbb{P}$-a.s.,
\begin{align}
\mathbb{E}[g(\tau,Z,\xi_u)|\mathcal{F}_t^{\xi,c}]&=\mathbb{E}[g(\tau,Z,\xi_u)\mathbb{I}_{\{\tau\leq u \}}|\mathcal{F}_t^{\xi,c}]+	\mathbb{E}[g(\tau,Z,\xi_u)\mathbb{I}_{\{ u<\tau \}}|\mathcal{F}_t^{\xi,c}]\nonumber\\
\nonumber\\
&=\mathbb{E}[g(\tau,Z,Z)\mathbb{I}_{\{\tau\leq u \}}|\mathcal{F}_t^{\xi,c}]+	\mathbb{E}[g(\tau,Z,\xi_u)\mathbb{I}_{\{ u<\tau \}}|\xi_u].
\end{align}
We remark that according to Theorem \ref{theorembayes}
\begin{align}
\mathbb{E}[g(\tau,Z,Z)\mathbb{I}_{\{\tau\leq u \}}|\mathcal{F}_t^{\xi,c}]=g(\tau,Z,Z)\mathbb{I}_{\{\xi_t=Z\}}+\sum\limits _{i=1}^{2}\dint_{t}^{u}g(r,z_i,z_i)\phi_{\zeta_{t}^{r,z_i}}(\xi_t)\mathbb{P}_{\tau}(\mathrm{d}r)p_i\mathbb{I}_{\{\xi_t\neq Z\}}.
\end{align}
On the other hand, from \eqref{eqargumentconditioanal}, \eqref{condlawresptauZ} and Theorem \ref{theorembayes}, for all $x\in \mathbb{R}$,
\begin{align}
\mathbb{E}[g(\tau,Z,\xi_u)&\mathbb{I}_{\{ u<\tau \}}|\xi_t=x]
=\dint_{\mathbb{R}}\dint_{u}^{+\infty}\mathbb{E}[g(r,z,\zeta_u^{r,z})\vert \zeta_t^{r,z}=x]\mathbb{P}(\tau \in \mathrm{d}r, Z\in \mathrm{d}z\vert \xi=x)\nonumber\\
&=\dint_{\mathbb{R}}\dint_{u}^{+\infty}\dint_{\mathbb{R}}g(r,z_i,y)\mathbb{P}(\zeta_u^{r,z_i}\in \mathrm{d}y\vert \zeta_t^{r,z_i}=x)\mathbb{P}(\tau \in \mathrm{d}r, Z\in \mathrm{d}z\vert \xi_t=x)\nonumber\\
&=\sum\limits _{i=1}^{2}\dint_{u}^{+\infty}\dint_{\mathbb{R}}g(r,z_i,y)\mathbb{P}(\zeta_u^{r,z_i}\in \mathrm{d}y\vert \zeta_t^{r,z_i}=x)\phi_{\zeta_{t}^{r,z_i}}(\xi_t)\mathbb{P}_{\tau}(\mathrm{d} r)p_i\mathbb{I}_{\{\xi_t\neq Z\}},
\end{align}
which finishes the proof.
\end{proof}
\begin{remark}\label{propzeta_ugivenzeta_t}
Let $0<t<u$ and $g$ be a bounded measurable function defined on $\mathbb{R}$. We have, $\mathbb{P}$-a.s.,
\begin{multline}
\mathbb{E}[g(\xi_u)|\mathcal{F}_t^{\xi,c}]=g(Z)\mathbb{I}_{\{\xi_t=Z\}}+\bigg[\sum\limits _{i=1}^{2}g(z_i)\dint_{t}^{u}\phi_{\zeta_{t}^{r,z_i}}(\xi_t)\mathbb{P}_{\tau}(\mathrm{d}r)p_i\\
+\sum\limits _{i=1}^{2}\dint_{u}^{+\infty}G_{t,u}(r,z_i,\xi_t)\phi_{\zeta_{t}^{r,z_i}}(\xi_t)\mathbb{P}_{\tau}(\mathrm{d} r)p_i\bigg]\mathbb{I}_{\{\xi_t\neq Z\}},\label{eqlawof(tau,Z,zeta_u)givenn-cordinateofzeta}
\end{multline}
where
\begin{equation}
G_{t,u}(r,z,x)=\dint_{\mathbb{R}}g(y)p\bigg(\dfrac{r-u}{r-t}(u-t),y,\dfrac{r-u}{r-t}x+\dfrac{u-t}{r-t}z\bigg)\mathrm{d}y.
\end{equation}
\end{remark}
\begin{remark}
\begin{enumerate}
\item[(i)]  Let $t > 0$, then, $\mathbb{P}$-a.s.,
\begin{equation}
\tau_t:=\mathbb{E}[\tau|\mathcal{F}_t^{\xi,c}]=\tau\,\mathbb{I}_{\{\tau\leq t\}}+\sum\limits _{i=1}^{2}\dfrac{\dint_{t}^{+\infty}r\dfrac{p(r-t,z_i-\xi_t)}{p(r,z_i)}\mathbb{P}_{\tau}(\mathrm{d}r)}{\Psi_{t}(\xi_t,z_1)p_1+\Psi_{t}(\xi_t,z_2)p_2}\,p_i\,\mathbb{I}_{\{t< \tau\}}.\label{eqtau_t}
\end{equation}
\begin{figure}[H]
\centering
\includegraphics[width=8cm]{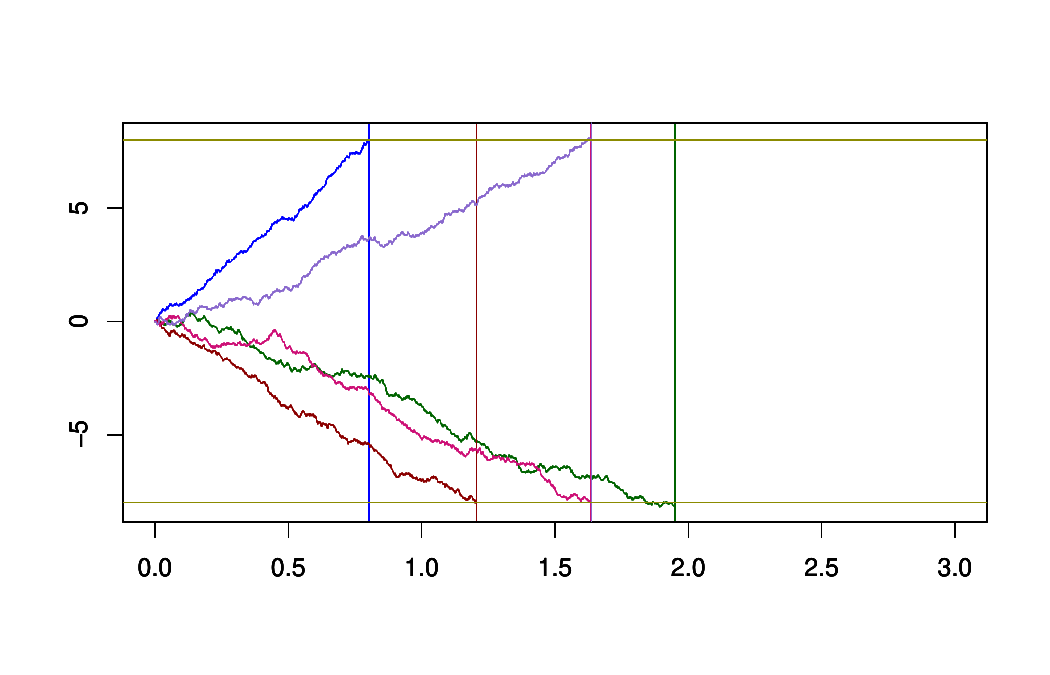}
\includegraphics[width=8cm]{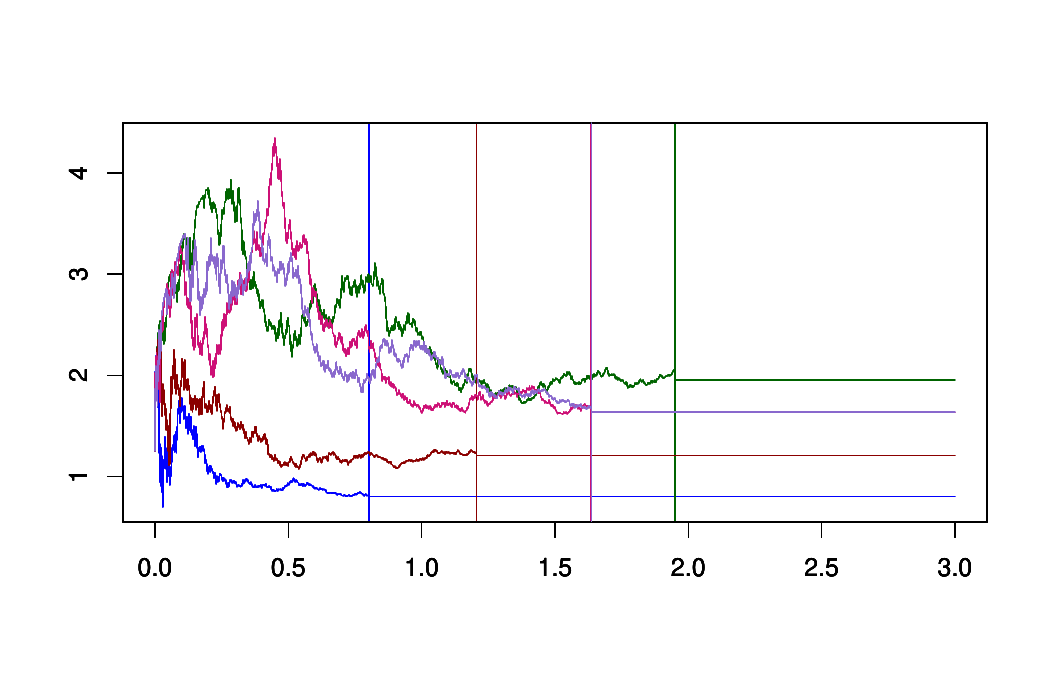}
\caption{The left hand side shows simulated paths of a Brownian bridge information process in the case where the length follows an exponential distribution with rate parameter $\lambda=0.8$, $z_2=-z_1=8$ and $p=0.3$. The right hand side shows simulated paths of $(\tau_t, t\geq 0)$.}
\end{figure}
\item[(ii)] Let $t > 0$, then, $\mathbb{P}$-a.s.,
\begin{equation}
Z_t:=\mathbb{E}[Z|\mathcal{F}_t^{\xi,c}]=Z\,\mathbb{I}_{\{\tau\leq t\}}+\bigg[z_1\bigg( 1+\dfrac{\Psi_{t}(\xi_t,z_2)p_2}{\Psi_{t}(\xi_t,z_1)p_1} \bigg)^{-1}
+z_2\bigg( 1+\dfrac{\Psi_{t}(\xi_t,z_1)p_1}{\Psi_{t}(\xi_t,z_2)p_2} \bigg)^{-1}\bigg]\mathbb{I}_{\{t< \tau\}}.\label{eqZ_t}
\end{equation}
\begin{figure}[H]
\centering
\includegraphics[width=8cm]{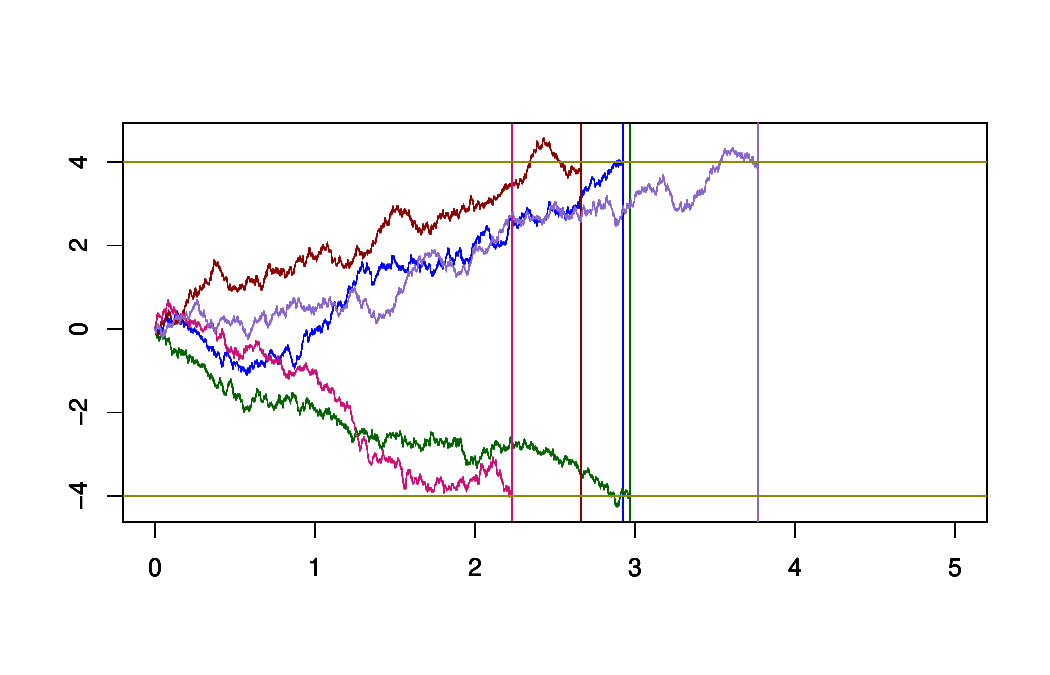}
\includegraphics[width=8cm]{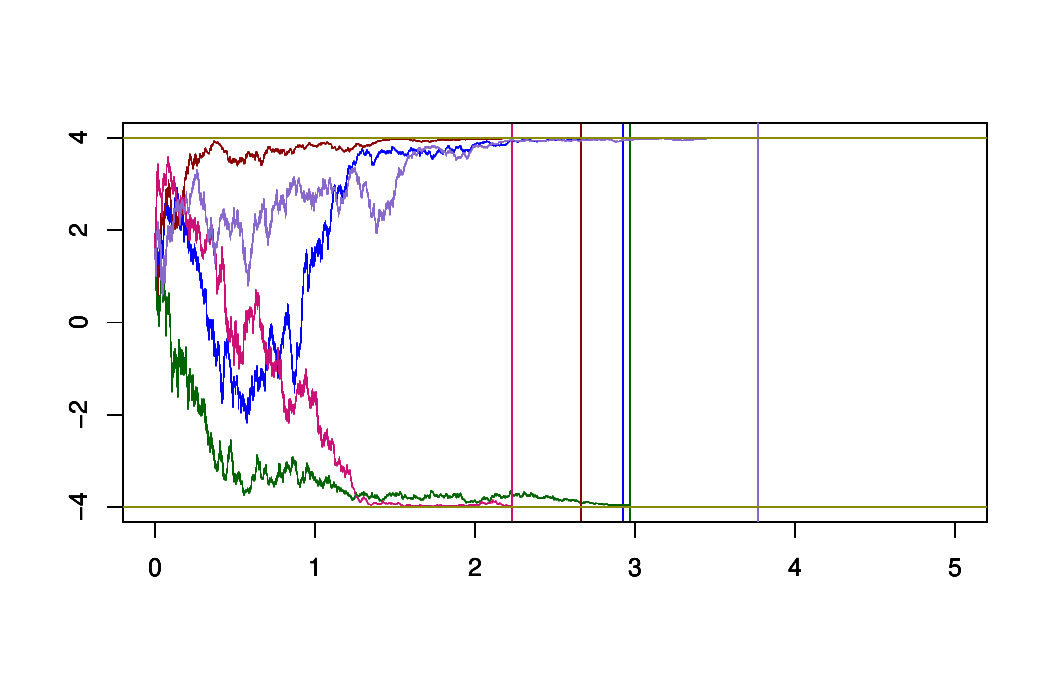}
\caption{The left hand side shows simulated paths of a Brownian bridge information process in the case where the length follows an exponential distribution with rate parameter $\lambda=0.3$, $z_2=-z_1=4$ and $p=0.7$. The right hand side shows simulated paths of $(Z_t, t\geq 0)$.}
\end{figure}
\item[(iii)] For $t < u$, the conditional expectation of $\xi_u$ given $\mathcal{F}_t^{\xi}$ is given by
\begin{multline}
\xi_{t,u}:=\mathbb{E}[\xi_u|\mathcal{F}_t^{\xi}]=\xi_t+\bigg[\sum\limits _{i=1}^{2}(z_i-\xi_t) \dfrac{\Psi_{t,u}(\xi_t,z_i)}{\Psi_{t}(\xi_t,z_1)p_1+\Psi_{t}(\xi_t,z_2)p_2}\,p_i\\
+\sum\limits _{i=1}^{2}(z_i-\xi_t) \dfrac{\dint_{u}^{+\infty}\dfrac{u-t}{r-t}\dfrac{p(r-t,z_i-\xi_t)}{p(r,z_i)}\mathbb{P}_{\tau}(\mathrm{d}r)}{\Psi_{t}(\xi_t,z_1)p_1+\Psi_{t}(\xi_t,z_2)p_2}\,p_i
\bigg]\mathbb{I}_{\{t< \tau\}}.\label{eqxi_t,u}
\end{multline}
\begin{figure}[H]
\centering
\includegraphics[width=8cm]{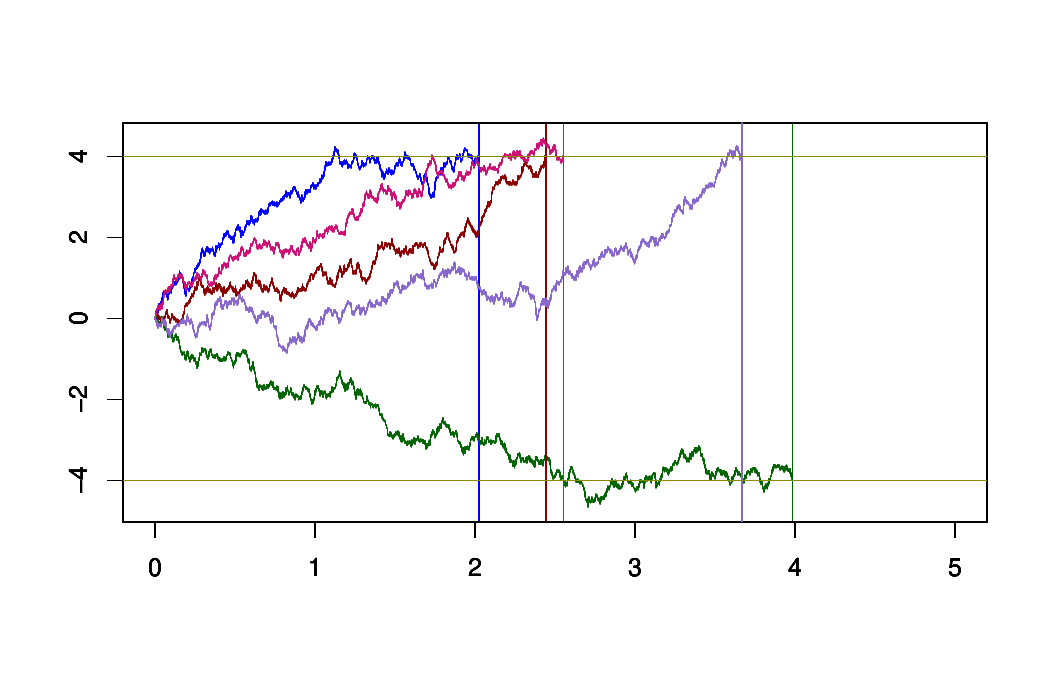}
\includegraphics[width=8cm]{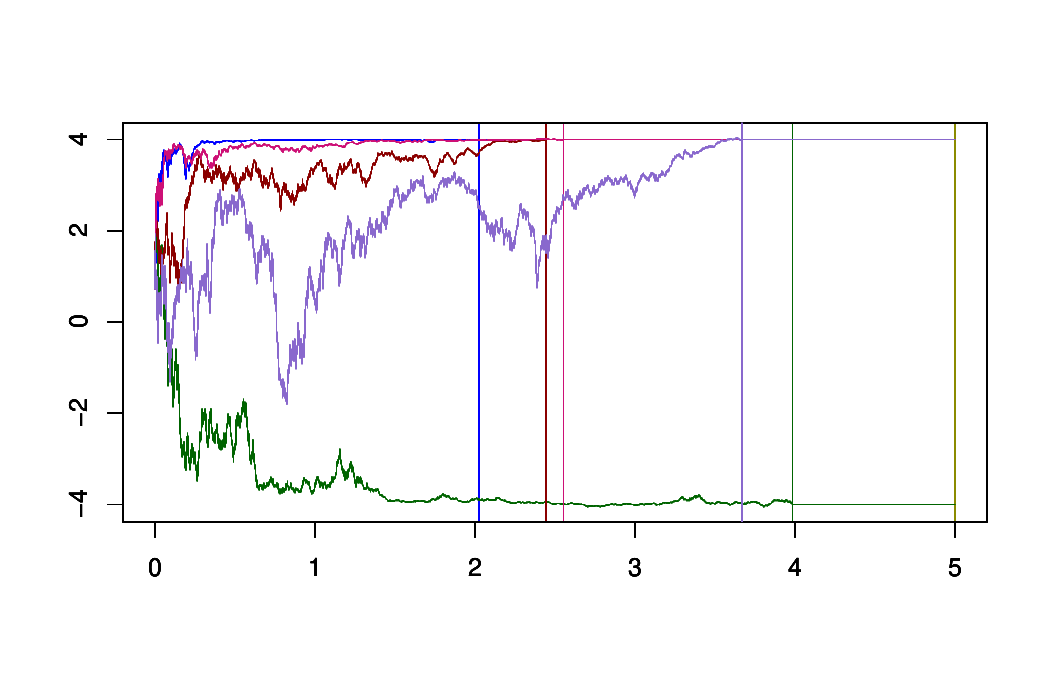}
\caption{The left hand side shows simulated paths of a Brownian bridge information process in the case where the length follows an exponential distribution with rate parameter $\lambda=0.4$, $z_2=-z_1=4$ and $p=0.7$. The right hand side shows simulated paths of $(\xi_{t,u}, 0\leq t\leq u)$ in case $u=5$.}
\end{figure}
\end{enumerate}
\end{remark}
In order to state the result of the right-continuity of the completed natural filtration associated to the process $\xi$, we need first to state the following auxiliary result:
\begin{lemma}\label{lemmatransitionfunctionrigthcontinuous}
Let $u$ be a strictly positive real number and $g$ be a bounded continuous function. Then, $\mathbb{P}$-a.s.,
\begin{enumerate}
\item[(i)] The function $t\longrightarrow \mathbb{E}[g(\xi_u)|\xi_t]$ is right-continuous on $]0,u]$.
\item[(ii)] If, in addition, $\mathbb{P}(\tau>\varepsilon)=1$ for some $\varepsilon>0$. Then the function $t\longrightarrow \mathbb{E}[g(\xi_u)|\xi_t]$ is right-continuous at $0$.
\end{enumerate}	
\end{lemma}
\begin{proof}
Let $t\in]0,\,u]$ and $(t_{n})_{n\in \mathbb{N}}$
be a decreasing sequence of strictly positive real numbers converging to $t$: that is $ t <...< t_{n+1} < t_{n} < u$, $t_{n} \searrow t$ as $n \longrightarrow +\infty $. Our goal is to show that 
\begin{equation}
\lim\limits_{n\rightarrow+\infty}\mathbb{E}[g(\xi_{u})\vert\xi_{t_n}]=\mathbb{E}[g(\xi_{u})\vert\xi_{t}].\label{eqmarkovlimitzeta}
\end{equation}
It follows from \eqref{eqlawof(tau,Z,zeta_u)givenn-cordinateofzeta} that we have, $\mathbb{P}$-a.s.,
\begin{align}
\mathbb{E}[g(\xi_u)\vert\xi_{t_n}]&=g(Z)\mathbb{I}_{\{\xi_{t_n}=Z\}}+\sum\limits_{i=1}^2g(z_i)\dint_{t_{n}}^{u}\phi_{\zeta_{t_{n}}^{r,z_i}}(\xi_{t_{n}})\mathbb{P}_{\tau}(\mathrm{d}r)p_i\;\mathbb{I}_{\{\xi_{t_n}\neq Z\}}\nonumber
\\&+\sum\limits_{i=1}^2\dint_{u}^{+\infty}G_{t_{n},u}(r,z_i,\xi_{t_{n}})\phi_{\zeta_{t_{n}}^{r,z_i}}(\xi_{t_{n}})\mathbb{P}_{\tau}(\mathrm{d}r)p_i\;\mathbb{I}_{\{\xi_{t_n}\neq Z\}}\nonumber\\
&=g(Z)\mathbb{I}_{\{\tau\leq t_{n}\}}+\sum\limits_{i=1}^2g(z_i)\dint_{t_{n}}^{u}\phi_{\zeta_{t_{n}}^{r,z_i}}(\xi_{t_{n}})\mathbb{P}_{\tau}(\mathrm{d}r)p_i\;\mathbb{I}_{\{t_{n}<\tau\}}\nonumber
\\&+\sum\limits_{i=1}^2\dint_{u}^{+\infty}G_{t_{n},u}(r,z_i,\xi_{t_{n}})\phi_{\zeta_{t_{n}}^{r,z_i}}(\xi_{t_{n}})\mathbb{P}_{\tau}(\mathrm{d}r)p_i\;\mathbb{I}_{\{t_{n}<\tau\}}.\label{eqlawof(tau,Z,zeta_u)givenn-zeta_n}
\end{align}
\begin{enumerate}
\item[(i)] The case $t>0$. 	Using \eqref{eqlawof(tau,Z,zeta_u)givenn-zeta_n} we see that the relation \eqref{eqmarkovlimitzeta} holds true if the
following two identities are satisfied, for all $i\in \{1,2\}$, $\mathbb{P}$-a.s., on $\{t<\tau\}$:

\begin{equation}
\lim\limits_{n\rightarrow+\infty}\dint_{t_{n}}^{u}\phi_{\zeta_{t_n}^{r,z_i}}(\xi_{t_n})\;\mathbb{P}_{\tau}(\mathrm{d}r)=\dint_{t}^{u}\phi_{\zeta_{t}^{r,z_i}}(\xi_t)\;\mathbb{P}_{\tau}(\mathrm{d}r)\label{eqfirst}
\end{equation}
\begin{align}
\lim\limits_{n\rightarrow+\infty}\dint_{u}^{+\infty}G_{t_n,u}(r,z_i,\xi_{t_n})\phi_{\zeta_{t_n}^{r,z_i}}(\xi_{t_n})\;\mathbb{P}_{\tau}(\mathrm{d}r)=\dint_{u}^{+\infty}G_{t,u}(r,z_i,\xi_{t})\phi_{\zeta_{t}^{r,z_i}}(\xi_t)\;\mathbb{P}_{\tau}(\mathrm{d}r).\label{eqsecond}
\end{align}
Note that the left-hand sides of \eqref{eqfirst} and \eqref{eqsecond} can be rewritten as
\begin{align*}
\dint_{t_{n}}^{u}\phi_{\zeta_{t_n}^{r,z_i}}(\xi_{t_n})\,\mathbb{P}_{\tau}(\mathrm{d}r)=
\dfrac{\dint_{t_{n}}^{u}\dfrac{p(r-t_n,z_i-\xi_{t_n})}{p(r,z_i)}\mathbb{P}_{\tau}(\mathrm{d}r)}{\sum\limits_{i=1}^2\dint_{t_{n}}^{+\infty}\dfrac{p(r-t_n,z_i-\xi_{t_n})}{p(r,z_i)}\mathbb{P}_{\tau}(\mathrm{d}r)p_i}=\dfrac{\Psi_{t_n,u}(\xi_{t_n},z_i)}{p_1\Psi_{t_n}(\xi_{t_n},z_1)+p_2\Psi_{t_n}(\xi_{t_n},z_2)}
\end{align*}
and
\begin{align*}
\dint_{u}^{+\infty}G_{t_n,u}(r,z_i,\xi_{t_n})\phi_{\zeta_{t_n}^{r,z_i}}(\xi_{t_n})\;\mathbb{P}_{\tau}(\mathrm{d}r)=
\dfrac{\dint_{u}^{+\infty}G_{t_n,u}(r,z_i,\xi_{t_n})\dfrac{p(r-t_n,z_i-\xi_{t_n})}{p(r,z_i)}\mathbb{P}_{\tau}(\mathrm{d}r)}{p_1\Psi_{t_n}(\xi_{t_n},z_1)+p_2\Psi_{t_n}(\xi_{t_n},z_2)}
\end{align*}
First let us observe that for all $i\in \{1,2\}$ the function 
\[
(t,r,x)\longmapsto \dfrac{p(r-t,z_i-x)}{p(r,z_i)}\mathbb{I}_{\{t<r\}}=\sqrt{\dfrac{r}{r-t}}\exp\bigg[-\frac{1}{2}\bigg(\dfrac{(z_i-x)^2}{r-t}-\dfrac{z_i^2}{r}\bigg)\bigg]\mathbb{I}_{\{t<r\}},
\]
defined on $\ensuremath{(0,+\infty)\times[0,+\infty)\times\mathbb{R}\backslash\{z_i\}}$
is continuous. Note that, $\mathbb{P}$-a.s., $$ \{t<\tau \}=\{\xi_t\neq Z\}=\{\xi_t\neq z_1\}\bigcap \{\xi_t\neq z_2\}.$$ 
Hence, $\mathbb{P}$-a.s. on $\left\{ t<\tau\right\} $,
we have for all $i\in \{1,2\}$,
\begin{equation}
\underset{n\rightarrow+\infty}{\lim}\,\dfrac{p(r-t_n,z_i-\xi_{t_n})}{p(r,z_i)}\mathbb{I}_{\{t_n<r\}}=\dfrac{p(r-t,z_i-\xi_t)}{p(r,z_i)}\mathbb{I}_{\{t<r\}}.\label{phixiconv}
\end{equation}
For any compact subset $\mathcal{K}$ of $(0,+\infty)\times\mathbb{R}\backslash\{z_i\}$
it yields 
\[
\underset{(t,x)\in\mathcal{K},r>0}{\sup}\,\dfrac{p(r-t,z_i-x)}{p(r,z_i)}\mathbb{I}_{\{t<r\}}<+\infty.
\]
It follows, $\mathbb{P}$-a.s., on $\left\{ t<\tau\right\} $ that for all $i\in \{1,2\}$
\begin{equation}
\underset{n\in\mathbb{N},r>t}{\sup}\,\dfrac{p(r-t_n,z_i-\xi_{t_n})}{p(r,z_i)}\mathbb{I}_{\{t_n<r\}}<+\infty.\label{varphiconvergence}
\end{equation}
We conclude assertion \eqref{eqfirst} from the Lebesgue dominated convergence theorem.\\
Now let us prove \eqref{eqsecond}. Recall that the
function $G_{t_{n},u}(r,z_i,\xi_{t_{n}})$ is given by 
$$G_{t_{n},u}(r,z_i,\xi_{t_{n}})=\dint_{\mathbb{R}} g(y)p\bigg(\dfrac{r-u}{r-t_n}(u-t_n),y,\dfrac{r-u}{r-t_n}\xi_{t_n}+\dfrac{u-t_n}{r-t_n}z_i\bigg)\mathrm{d}y.$$
Note that the function
$$ y\longrightarrow p\bigg(\dfrac{r-u}{r-t_n}(u-t_n),y,\dfrac{r-u}{r-t_n}\xi_{t_n}+\dfrac{u-t_n}{r-t_n}z_i\bigg) $$
is a density on $\mathbb{R}$ for all $n$. Since $g$ is bounded, we deduce that $G_{t_{n},u}(r,z,\xi_{t_{n}})$ is bounded. Moreover we obtain from the weak convergence of Gaussian measures that  
\[
\underset{n\rightarrow+\infty}{\lim}\, G_{t_{n},u}(r,z,\xi_{t_{n}})=G_{t,u}(r,z,\xi_{t}),
\] 
combining the fact that $G_{t_{n},u}(r,z,\xi_{t_{n}})$ is bounded, \eqref{phixiconv}  and \eqref{varphiconvergence} , the assertion \eqref{eqsecond} is derived from the Lebesgue dominated convergence theorem.
\item[(ii)] The case $t=0$. Let us prove the relation \eqref{eqmarkovlimitzeta} under the assumption: there exists $\varepsilon>0$ such that $\mathbb{P}(\tau>\varepsilon)=1$. From \eqref{eqmarkovlimitzeta}, it is sufficient to verify that
\begin{equation}\label{equationlimitmarkovstep2t=0zeta}
\lim\limits_{n \rightarrow+\infty} \mathbb{E}[g(\xi_{u})\vert \xi_{t_{n}}]=\mathbb{E}[g(\xi_{u})\vert \xi_{0}],~~\mathbb{P}\text{-a.s.}
\end{equation}
we have, 
\begin{align}
\mathbb{E}[g(\xi_{u})]&=F(u)\mathbb{E}[g(Z)]+\dint_{\mathbb{R}}\dint_{u}^{+\infty}\dint_{\mathbb{R}}g(y)p\bigg(\dfrac{u(r-u)}{r},y,\dfrac{u}{r}z\bigg)\,\mathrm{d}y\,\mathbb{P}_{\tau}(\mathrm{d}r)\mathbb{P}_{Z}(\mathrm{d}z)\nonumber\\
&=F(u)\,\mathbb{E}[g(Z)]\,+\sum\limits_{i=1}^2\dint_{u}^{+\infty}\dint_{\mathbb{R}}g(y)p\left(\dfrac{u(r-u)}{r},y,\dfrac{u}{r}z_i\right)\mathrm{d}y\,\mathbb{P}_{\tau}(\mathrm{d}r)p_i.\label{eqEg(zeta_u)}
\end{align}
Without loss of generality we assume that
$t_{n}< \varepsilon $ for all $n\in \mathbb{N}$. Due to \eqref{eqlawof(tau,Z,zeta_u)givenn-zeta_n}, the fact that there exists $\varepsilon>0$ such that $\mathbb{P}(\tau>\varepsilon)=1$ and \eqref{eqEg(zeta_u)}, the statement \eqref{equationlimitmarkovstep2t=0zeta} will be proven once we prove that, $\mathbb{P}$-a.s., for all $i\in \{1,2\}$
\begin{equation}
\lim\limits_{n\rightarrow+\infty}\dint_{t_{n}}^{u}\phi_{\zeta_{t_n}^{r,z_i}}(\xi_{t_n})\;\mathbb{P}_{\tau}(\mathrm{d}r)=F(u),\label{eqfirststep2}
\end{equation}
\begin{align}
\lim\limits_{n\rightarrow+\infty}\dint_{u}^{+\infty}G_{t_n,u}(r,z_i,\xi_{t_n})\phi_{\zeta_{t_n}^{r,z_i}}(\xi_{t_n})\;\mathbb{P}_{\tau}(\mathrm{d}r)=\dint_{u}^{+\infty}\dint_{\mathbb{R}}g(y)p\bigg(\dfrac{u(r-u)}{r},y,\dfrac{u}{r}z_i\bigg)\mathrm{d}y\,\mathbb{P}_{\tau}(\mathrm{d}r).\label{eqsecondstep2}
\end{align}

Moreover, we have for all $r>\varepsilon$, for all $i\in \{1,2\}$,
\begin{equation}
\dfrac{p(r-t_n,z_i-\xi_{t_n})}{p(r,z_i)}\leq \sqrt{\dfrac{r}{r-t_n}}\exp\bigg[\dfrac{z_i^2}{2r}\bigg]\leq \sqrt{\dfrac{\varepsilon}{\varepsilon-t_1}}\exp\bigg[\dfrac{z_i^2}{2\varepsilon}\bigg].\label{eqbound}
\end{equation}
combining the fact that $g$ and $G_{t_{n},u}(r,z,\xi_{t_{n}})$ are bounded and inequality \eqref{eqbound} the assertions \eqref{eqfirststep2} and \eqref{eqsecondstep2} can be derived from the Lebesgue dominated convergence theorem.
\end{enumerate}
\end{proof}
We are now able  to state another main result of this section, namely the right-continuity of the completed natural filtration of the  Brownian bridge information process. 
\begin{theorem}\label{thmrightcontinuityfiltration}
The  filtration $\mathbb{F}^{\xi,c}$ satisfies the usual conditions of right-continuity and completeness.
\end{theorem}
\begin{proof}
It is sufficient to prove that for every bounded
$\mathcal{F}_{t+}^{\xi,c}$-measurable $Y$ we have, $\mathbb{P}$-a.s.,
\begin{equation}
\mathbb{E}[Y|\mathcal{F}_{t+}^{\xi,c}]=\mathbb{E}[Y|\mathcal{F}_{t}^{\xi,c}].\label{eqrightcontinuityfiltration}
\end{equation}
This is an immediate consequence of the Markov property of $\xi$ with respect to $\mathbb{F}^{\xi,c}_+$. Let us first prove that $\xi$ is an $\mathbb{F}_+^{\xi,c}$-Markov process, i.e.,
\begin{equation}
\mathbb{E}[g(\xi_{u})\vert\mathcal{F}_{t+}^{\xi}]=\mathbb{E}[g(\xi_{u})\vert\xi_t], \mathbb{P}\text{-a.s.},\label{eqMarkovF_t+}
\end{equation}
for all $t<u$ and for every bounded measurable function $g$.
Throughout the proof we can assume without loss of
generality that the function $g$ is continuous and bounded. Let $(t_{n})_{n\in \mathbb{N}}$
be a decreasing sequence of strictly positive real numbers converging to $t$: that is $ t <\ldots< t_{n+1} < t_{n} < u$, $t_{n} \searrow t$ as $n \rightarrow +\infty $. Since $g$ is bounded, $\mathcal{F}_{t+}^{\xi,c}=\underset{n}{\cap}\mathcal{F}_{t_{n}}^{\xi,c}$ and $\xi$ is an $\mathbb{F}^{\xi,c}$-Markov process we have,  $\mathbb{P}$-a.s., 
\begin{align}
\mathbb{E}[g(\xi_{u})\vert\mathcal{F}_{t+}^{\xi}]&=\lim\limits_{n\rightarrow+\infty}\mathbb{E}[g(\xi_{u})|\mathcal{F}^{\xi,c}_{t_n}]\nonumber\\
&=\lim\limits_{n\rightarrow+\infty}\mathbb{E}[g(\xi_{u})\vert\xi_{t_n}].
\end{align}
In order to prove \eqref{eqMarkovF_t+} we need to show that for all $t\geq 0$, $\mathbb{P}$-a.s.,
\begin{equation}
\lim\limits_{n\rightarrow+\infty}\mathbb{E}[g(\xi_{u})\vert\xi_{t_n}]=\mathbb{E}[g(\xi_{u})\vert\xi_{t}].\label{eqthetransitionfuncionrigthcontinuous}
\end{equation}
According to Lemma \ref{lemmatransitionfunctionrigthcontinuous}, the only case that remains to be proved is $t=0$ with $\mathbb{P}(\tau>0)=0$. To see this, it is sufficient to show that $\mathcal{F}_{0+}^{\xi,c}$ is $\mathbb{P}$-trivial. This amounts to proving that $\mathcal{F}_{0+}^{\xi}$ is $\mathbb{P}$-trivial, since $\mathcal{F}_{0+}^{\xi,c}=\mathcal{F}_{0+}^{\xi} \vee \mathcal{N}_{P}$. For this purpose, let $\varepsilon > 0$ be fixed and consider the stopping time $\tau_{\varepsilon}=\tau \vee \varepsilon$. We define the process $\xi_{t}^{\tau_{\varepsilon}}$ by $$\left\lbrace\xi_{t}^{\tau_{\varepsilon}};\,t\geq 0\right\rbrace:=\left\lbrace\xi_{t}^{r} \vert_{r=\tau \vee \varepsilon};\,t\geq 0\right\rbrace.$$  
First observe that the sets $\{\tau_{\varepsilon}>\varepsilon\}=\{\tau>\varepsilon\}$ are equal  and therefore the following equality of processes holds $$\xi_{\cdot}^{\tau_{\varepsilon}}\mathbb{I}_{\{\tau>\varepsilon\}}=\xi_{\cdot}\;\mathbb{I}_{\{\tau>\varepsilon\}}.$$
Then for each $A\in \mathcal{F}_{0+}^{\xi}$ there exists $B\in \mathcal{F}_{0+}^{\xi^{\tau_{\varepsilon}}}$ such that  
$$A\cap\{\tau>\varepsilon\}=B\cap \{\tau>\varepsilon\}.$$
As $\mathbb{P}(\tau_{\varepsilon}>\varepsilon/2)=1$,  according to Lemma \ref{lemmatransitionfunctionrigthcontinuous}, $\mathcal{F}_{0+}^{\xi^{\tau_{\varepsilon}}}$ is $\mathbb{P}$-trivial. Hence, $\mathbb{P}(B)$ is equal to $1$ or $0$. Consequently, we obtain 
$$\mathbb{P}(A\cap\{\tau>\varepsilon\})=0\text{\,\,or\,\,}\mathbb{P}(A\cap\{\tau>\varepsilon\})=\mathbb{P}(\tau>\varepsilon).$$
Now if $\mathbb{P}(A)>0$, then there exists $\varepsilon > 0$ such that $\mathbb{P}(A\cap \{\tau> \varepsilon\})>0$. Therefore, for all $0<\varepsilon'\leq \varepsilon$ we have 
$$\mathbb{P}(A\cap\{\tau>\varepsilon'\})=\mathbb{P}(\tau>\varepsilon').$$ 
Passing to the limit as $\varepsilon'$ goes to $ 0$ yields $\mathbb{P}(A\cap\{\tau>0\})=\mathbb{P}(\tau>0)=1$. It follows that $\mathbb{P}(A)=1$, which ends the proof.
\end{proof}
Our purpose now is to derive the semi-martingale property of $\xi$ with respect to its own filtration $\mathbb{F}^{\xi,c}$.
\begin{theorem}\label{thmdecompositionsemimartingaletauZ}
The semi-martingale decomposition of $\xi$ in its natural filtration $\mathbb{F}^{\xi,c}$ is given by
\begin{align}
\xi_{t}&=I_t+\dint_{0}^{t}\mathbb{E}\bigg[\dfrac{Z-\xi_s}{\tau-s}\mathbb{I}_{\{s<\tau\}}\vert\xi_s\bigg]\mathrm{d}s\nonumber\\
&=I_t+\dint_0^{t\wedge \tau}\sum\limits_{i=1}^2 \dfrac{(z_i-\xi_s)p_i}{\Psi_{s}(\xi_s,z_1)p_1+\Psi_{s}(\xi_s,z_2)p_2}\dint_{s}^{+\infty}\dfrac{1}{r-s}\dfrac{p(r-s,z_i-\xi_s)}{p(r,z_i)}\mathbb{P}_{\tau}(\mathrm{d}r)\,\mathrm{d}s,\label{equationdecompositionsemitauZ}
\end{align}
where the process  $(I_t, t\geq0)$ is an $\mathbb{F}^{\xi,c}$-Brownian motion stopped at $\tau$ and the function $\Psi_{s}$ is defined by \eqref{eqPsi_s}.
\end{theorem}
\begin{proof}
From the representation \eqref{eqsemimartingalezeta^r,zonR+} we obtain that
\begin{equation}
b_t=\xi_{t}-\int_{0}^{t}\dfrac{Z-\xi_{s}}{\tau-s}\mathbb{I}_{\{s<\tau\}}\mathrm{d}s,\label{eqsemimartingaletildezeta^r,z}
\end{equation}
where the process $b$ is defined as follows:
$$b_{t}(\omega):=b_t^{r,z}(\omega)\vert_{r=\tau(\omega)}^{z=Z(\omega)},$$
for $(t,\omega) \in \mathbb{R}_{+} \times \Omega$, where $(b^{r,z}_t, t\geq 0)$ is an $\mathbb{F}^{\zeta^{r,z}}$-Brownian motion stopped at $r$. Since the random variable $Z$ is integrable, for all $t\in \mathbb{R}_+$, we have 
\begin{equation}
\mathbb{E}\bigg[\dint_{0}^{t}\bigg|\dfrac{Z-\xi_s}{\tau-s}\bigg|\mathbb{I}_{\{s<\tau\}}\mathrm{d}s\bigg]<\infty.\label{eqintegrable}
\end{equation}
To see that, by using the formula of total probability and that $\mathbb{E}[\vert \zeta_{s}^{r,0} \vert ]=\sqrt{\dfrac{2}{\pi}}\sqrt{\dfrac{s(r-s)}{r}}$, we have for all $t\geq 0$,
\begin{align}
\mathbb{E}\bigg[\dint_{0}^{t}\,\bigg\vert\dfrac{Z-\xi_s}{\tau-s}\bigg|&\mathbb{I}_{\{s<\tau\}}\,\mathrm{d}s\bigg]=\dint_{0}^{t}\dint_{\mathbb{R}}\dint_{0}^{+\infty}\mathbb{E}\bigg[\dfrac{|Z-\xi_{s}|}{\tau-s}\,\mathbb{I}_{\{s<\tau\}}\bigg|\tau=r,Z=z\bigg]\mathbb{P}_{\tau}(\mathrm{d}r)\mathbb{P}_{Z}(\mathrm{d}z)\mathrm{d}s\nonumber\\
&=\dint_{0}^{t}\dint_{\mathbb{R}}\dint_{s}^{+\infty}\mathbb{E}\bigg[\dfrac{|z-\zeta_{s}^{r,z}|}{r-s}\bigg]\mathbb{P}_{\tau}(\mathrm{d}r)\mathbb{P}_{Z}(\mathrm{d}z)\mathrm{d}s\nonumber\\
&=\dint_{0}^{t}\dint_{\mathbb{R}}\dint_{s}^{+\infty}\mathbb{E}\bigg[\dfrac{|z\dfrac{r-s}{r}-\zeta_{s}^{r,0}|}{r-s}\bigg]\mathbb{P}_{\tau}(\mathrm{d}r)\mathbb{P}_{Z}(\mathrm{d}z)\mathrm{d}s\nonumber\\
&\leq \dint_{0}^{t}\dint_{\mathbb{R}}\dint_{s}^{+\infty}\bigg(\dfrac{|z|}{r}+\sqrt{\dfrac{2}{\pi}} \sqrt{\dfrac{s}{r(r-s)}}\bigg)\mathbb{P}_{\tau}(\mathrm{d}r)\mathbb{P}_{Z}(\mathrm{d}z)\mathrm{d}s\nonumber\\
&=\mathbb{E}[|Z|]\dint_{0}^{+\infty}\dint_{0}^{t\wedge r}\dfrac{1}{r}\mathrm{d}s\mathbb{P}_{\tau}(\mathrm{d}r)+\sqrt{\dfrac{2}{\pi}} \dint_{0}^{+\infty}\dint_{0}^{t\wedge r}\sqrt{\dfrac{s}{r(r-s)}}\mathrm{d}s\mathbb{P}_{\tau}(\mathrm{d}r).\label{eqintegral}
\end{align}
It is clear that
\begin{equation*}
\dint_{0}^{+\infty}\dint_{0}^{t\wedge r}\dfrac{1}{r}\mathrm{d}s\mathbb{P}_{\tau}(\mathrm{d}r)\leq \dint_{0}^{+\infty}\mathbb{P}_{\tau}(\mathrm{d}r)=1.
\end{equation*}
We split the second integral on the right-hand side of \eqref{eqintegral} into two integrals:
\begin{align*}
\dint_{0}^{+\infty}\dint_{0}^{t\wedge r}\sqrt{\dfrac{s}{r(r-s)}}\mathrm{d}s\mathbb{P}_{\tau}(\mathrm{d}r)&= 	\dint_{0}^{t}\dint_{0}^{t\wedge r}\sqrt{\dfrac{s}{r(r-s)}}\mathrm{d}s\mathbb{P}_{\tau}(\mathrm{d}r)+	\dint_{t}^{+\infty}\dint_{0}^{t\wedge r}\sqrt{\dfrac{s}{r(r-s)}}\mathrm{d}s\mathbb{P}_{\tau}(\mathrm{d}r).
\end{align*}
For the first integral we see that
\begin{align*}
\dint_{0}^{t}\dint_{0}^{t\wedge r}\sqrt{\dfrac{s}{r(r-s)}}\mathrm{d}s\mathbb{P}_{\tau}(\mathrm{d}r)&\leq\dint_{0}^{t}\dint_{0}^{r}\dfrac{1}{\sqrt{r-s}}\mathrm{d}s\mathbb{P}_{\tau}(\mathrm{d}r)=\dint_{0}^{t}2\sqrt{r}\mathbb{P}_{\tau}(\mathrm{d}r)\leq 2\sqrt{t}.
\end{align*}
The second integral can be estimated as follows:
\begin{align*}
\dint_{t}^{+\infty}\dint_{0}^{t\wedge r}\sqrt{\dfrac{s}{r(r-s)}}\mathrm{d}s\mathbb{P}_{\tau}(\mathrm{d}r)&\leq \sqrt{t} 	\dint_{t}^{+\infty}\dfrac{1}{\sqrt{r}}\dint_{0}^{t}\dfrac{1}{\sqrt{r-s}}\mathrm{d}s\mathbb{P}_{\tau}(\mathrm{d}r)\leq 2\sqrt{t}.
\end{align*}
Now let us consider the filtration 
\begin{equation}
\mathbb{H}=\left(\mathcal{H}_t:=\mathcal{F}^{\xi,c}_{t}\vee \sigma(\tau,Z),\,\, t\geq 0\right),
\end{equation}	
which is equal to the initial enlargement of the filtration $\mathbb{F}^{\xi,c}$ by the $\sigma$-algebra $\sigma(\tau,Z)$. From \eqref{eqintegrable}, the process  $b$ is well-defined. Moreover, it is a Brownian motion stopped at $\tau$ with respect to $\mathbb{H}$. Indeed, 
It is clear that the process $b$ is continuous and $\mathbb{H}$-adapted. In order to prove that it is a $\mathbb{H}$-Brownian motion stopped at $\tau$, it suffices to prove that the process $b$ and the process $X$ defined by $X_t := b^2_t - (t \wedge \tau)$, $t \geq 0$,
are both $\mathbb{H}$-martingales. Since $(b^{r,z}_t, t\geq 0)$ is an $\mathbb{F}^{\zeta^{r,z}}$-Brownian motion stopped at $r$, the process defined by $X^{r,z}:=(b^{r,z}_t-t\wedge r, t\geq 0)$ is an $\mathbb{F}^{\zeta^{r,z}}$-martingale. Moreover, we obtain for any $0<t_{1}<t_{2}<...<t_{n}=t$, $n\in \mathbb{N}^*$, $h \geq 0$ and for bounded Borel functions $g$, that
\begin{align*}
\mathbb{E}\left[(b_{t+h}-b_{t})g(\xi_{t_{1}},\ldots,\xi_{t_{n}},\tau,Z)\right] &=\dint_{\mathbb{R}}\dint_{(0,+\infty)}\mathbb{E}[(b_{t+h}^{r,z}-b_{t}^{r,z})g(\zeta_{t_{1}}^{r,z},\ldots,\zeta_{t_{n}}^{r,z},r,z)]\mathbb{P}_{\tau}(\mathrm{d}r)\mathbb{P}_{Z}(\mathrm{d}z)\\
&=0,
\end{align*}
and
\begin{align*}
\mathbb{E}\left[(X_{t+h}-X_{t})g(\xi_{t_{1}},\ldots,\xi_{t_{n}},\tau,Z)\right]  &=\dint_{\mathbb{R}}\dint_{(0,+\infty)}\mathbb{E}[(X_{t+h}^{r,z}-X_{t}^{r,z})g(\zeta_{t_{1}}^{r,z},\ldots,\zeta_{t_{n}}^{r,z},r,z)]\mathbb{P}_{\tau}(\mathrm{d}r)\mathbb{P}_{Z}(\mathrm{d}z)\\
&=0.
\end{align*}
The desired result follows by a standard monotone class argument. A well known result of filtering theory, see Proposition
2.30, p. 33 in \cite{BCfiltering} tells us that the process $I$  given by
\begin{equation}
I_t=\xi_{t}-\dint_{0}^{t}\mathbb{E}\bigg[\dfrac{Z-\xi_s}{\tau-s}\mathbb{I}_{\{s<\tau\}}\bigg|\mathcal{F}_s^{\xi}\bigg]\mathrm{d}s,\; t\geq 0,\label{eqinnovationprocesswithexpectation}
\end{equation}
is an $\mathbb{F}^{\xi,c}$-Brownian motion stopped at $\tau$.
From Theorem \ref{thmrightcontinuityfiltration}, the  filtration $\mathbb{F}^{\xi,c}$ satisfies the usual conditions of right-continuity and completeness. The semi-martingale decomposition of $\xi$ with respect to its own filtration $\mathbb{F}^{\xi,c}$ is given by \eqref{eqinnovationprocesswithexpectation}.
Therefore, it remains to compute the conditional expectation of $\dfrac{Z-\xi_s}{\tau-s}\mathbb{I}_{\{s<\tau\}}$ with respect to $\mathcal{F}_s^{\xi,c}$. Indeed, using  \eqref{eqlawoftauZgivenfzeta} we have,
$\mathbb{P}$-a.s.,
\[
\begin{array}{ll}
\mathbb{E}\left[\dfrac{Z-\xi_s}{\tau-s}\mathbb{I}_{\{s<\tau\}}\vert\mathcal{F}_{s}^{\xi,c}\right]& =\sum\limits_{i=1}^2 \dfrac{(z_i-\xi_s)p_i}{\Psi_{s}(\xi_s,z_1)p_1+\Psi_{s}(\xi_s,z_2)p_2}\dint_{s}^{+\infty}\dfrac{1}{r-s}\dfrac{p(r-s,z_i-\xi_s)}{p(r,z_i)}\mathbb{P}_{\tau}(\mathrm{d}r).
\end{array}
\]
Hence we derive the canonical decomposition \eqref{equationdecompositionsemitauZ} of $\xi$ as a semi-martingale with respect to its
own filtration $\mathbb{F}^{\xi,c}$.
\end{proof}
\textbf{Acknowldgements:}
I would like to express particular thanks to the anonymous referees for the constructive comments that greatly improved the manuscript. I would also like to express my deep gratitude to Professor Mohamed Erraoui and Professor Astrid Hilbert  for their guidance, enthusiastic encouragement and for many stimulating conversations.  My grateful thanks are also extended to Professor Monique Jeanblanc for the many helpful comments during the preparation of the paper. I gratefully acknowledge the financial support by an Erasmus+ International Credit Mobility exchange project coordinated by Linnaeus University.


\begin{thebibliography}{99}
\bibitem{A}
Alili, L. Canonical decompositions of certain generalized Brownian bridges. Electron. Comm. Probab. 7 (2002), 27--36.
\bibitem{B}
Back, K. Insider trading in continuous time. Rev. Financ. Stud. 5 (1992), 387--409.
\bibitem{BCfiltering}
Bain, A.; Crisan, D. \textit{Fundamentals of stochastic filtering}. Stochastic Modelling and Applied Probability, 60. Springer, New York, (2009).	
\bibitem{BBE}
Bedini, M. L.; Buckdahn, R.; Engelbert, H. J. Brownian bridges on random intervals. Theory Probab. Appl. 61 (2017), no. 1, 15--39.
\bibitem{BG}
Blumenthal, R. M.; Getoor, R. K. \textit{ Markov processes and potential theory}. Academic Press, New York-London (1968).
\bibitem{BJ}
Boogert, A.; de Jong, C. Gas storage valuation using a Monte Carlo method. J. Derivatives 15 (2008), 81--98.
\bibitem{BS}
Brennan, M. J.; Schwartz, E. S. Arbitrage in stock index futures. Journal of Business. 63 (1990), 7--31.
\bibitem{BHM2007}
Brody, D. C.; Hughston, L. P.; Macrina, A. Beyond hazard rates: a new framework to credit-risk modelling. In Advances in mathematical finance (eds M. Fu, R. Jarrow, J.-Y. J. Yen and R. Elliott), pp. 231–257, (2007).
\bibitem{BHM2008}
Brody, D. C.; Hughston, L. P.; Macrina, A. Information-based asset pricing. Int. J. Theor. Appl. Finance 11 (2008), no. 1, 107--142.
\bibitem{BHM}
Brody, D.C.; Hughston, L.P.; and  Macrina, A. Dam rain and cumulative gain. Proc. R. Soc. Lond. Ser. A Math. Phys. Eng. Sci. 464 (2008), no. 2095, 1801--1822.
\bibitem{CL}
Carmona, R.; Ludkovski, M. Valuation of energy storage: An optimal switching approach. Quant. Finance 10 (2010), no. 4, 359--374.
\bibitem{CF}
Chen, Z.; Forsyth, P. A. A semi-lagrangian approach for natural gas storage valuation and optimal operations. SIAM J. Sci. Comput. 30 (2007), no. 1, 339--368.
\bibitem{EV}
Ekstr\"om, E.; Vaicenavicius, J. Optimal stopping of a Brownian bridge with unknown pinning point. Stochastic Process. Appl. 130 (2020), no. 2, 806--823.
\bibitem{EW}
Ekstr\"om, E.; Wanntorp, H. Optimal stopping of a Brownian bridge. J. Appl. Probab. 46 (2009), no. 1, 170--180.
\bibitem{EY}
Emery, M.; Yor, M. A parallel between Brownian bridges and gamma bridges. Publ. Res. Inst. Math. Sci. 40 (2004), no. 3, 669--688.
\bibitem{EHL}
Erraoui, M.; Hilbert, A.; Louriki, M. Bridges with random length: gamma case. J. Theoret. Probab. 33 (2020), no. 2, 931--953.
\bibitem{EHL(Levy)}
Erraoui, M.; Hilbert, A.; Louriki, On a L\'evy process pinned at random time. Forum Math. 33 (2021), no. 2, 397--417.
\bibitem{EL}
Erraoui, M.; Louriki, M. Bridges with random length: Gaussian-Markovian case. Markov Process. Related Fields 24 (2018), no. 4, 669--693.
\bibitem{FG}
Fitzsimmons, P. J.; Getoor, R. K. Occupation time distributions for L\'evy bridges and excursions. Stochastic Process. Appl. 58 (1995), no. 1, 73--89.
\bibitem{FPY}
Fitzsimmons P.J.;  Pitman J. ;  Yor M. Markovian bridges: Construction, palm interpretation, and splicing. Seminar on Stochastic Processes, 1992 (Seattle, WA, 1992), 101--134, Progr. Probab., 33, Birkhäuser Boston, Boston, MA, (1993).
\bibitem{GSV}
Gasbarra, D.; Sottinen, T.; Valkeila, E. Gaussian bridges. Stochastic analysis and applications, 361--382, Abel Symp., 2, Springer, Berlin, (2007).
\bibitem{H}
Holland, A. Optimization of injection/withdrawal schedules for natural gas storage facilities. In
Twenty-seventh SGAI International Conference on Artificial Intelligence, (2007).
\bibitem{HHM}
Hoyle, E.;  Hughston, L.P.;  Macrina, A. L\'evy random bridges and the modelling of financial information. Stochastic Process. Appl. 121 (2011), no. 4, 856--884.
\bibitem{K}
Kyle, A. Continuous auctions and insider trading. Econometrica. 53 (1985), 1315--1335.
\bibitem{R}
Ros\`en, B. Limit theorems for sampling from finite populations. Ark. Mat. 5 (1965), 383--424.
\bibitem{S}
Shiryaev, Albert N. \textit{Probability. 1}.  Springer, New York, Third edition, (2016).
\end{thebibliography}
\end{document}